\documentclass[11pt]{amsart}
\usepackage{latexsym}
\usepackage{amssymb}
\usepackage{amscd}
\usepackage{mathrsfs}
\renewcommand\a{\alpha}

\renewcommand\d{\delta}
\newcommand\la{\lambda}
\newcommand\z{\zeta}

\renewcommand\th{\theta}

\newcommand\s{\sigma}
\newcommand\x{\chi}
\newcommand\f{\phi}
\newcommand\vf{\varphi}

\renewcommand\r{\rho}

\newcommand\Om{\Omega}
\newcommand\w{\omega}

\newcommand\vD{\varDelta}

\newcommand\vL{\varLambda}

\newcommand{\vT}{\varTheta}

\newcommand{\GG}{\mathbb G}

\newcommand\SA{\mathscr{A}}

\newcommand\SC{\mathscr{C}}

\newcommand\SE{\mathscr{E}}

\newcommand\SM{\mathscr{M}}
\newcommand\SN{\mathscr{N}} 

\newcommand\SP{\mathscr{P}}

\newcommand\SH{\mathscr{H}}

\newcommand\ScS{\mathscr{S}}
\newcommand\SW{\mathscr{W}}

\newcommand\Ql{\bar{\mathbf Q}_l}

\newcommand\BP{\mathbf P}
\newcommand\BQ{\mathbf Q}
\newcommand\BF{\mathbf F}
\newcommand\BC{\mathbf C}

\newcommand\BZ{\mathbf Z}

\newcommand\BV{\mathbf V}

\newcommand\Bm{\mathbf m}

\newcommand\Bs{\mathbf s}

\newcommand\Bk{\mathbf k}

\newcommand\Bd{\mathbf d}
\newcommand\Bt{\mathbf t}

\newcommand\BLa{\boldsymbol\varLambda}

\newcommand\Bla{\boldsymbol\lambda}

\newcommand\Bmu{\boldsymbol\mu}

\newcommand\Bxi{\boldsymbol{\xi}}
\newcommand\BOm{\boldsymbol{\Omega}}

\newcommand\iv{^{-1}}
\newcommand\wh{\widehat}
\newcommand\wt{\widetilde}
\newcommand\wg{^{\wedge}}

\newcommand\ol{\overline}

\newcommand\hra{\hookrightarrow}
\newcommand\lra{\leftrightarrow}

\newcommand\lv{\prec}
\newcommand\gv{\succ}
\newcommand\lve{\preceq}
\newcommand\gve{\succeq}

\newcommand\IC{\operatorname{IC}}

\newcommand\supp{\operatorname{supp}\,}

\newcommand\Tr{\operatorname{Tr}\,}
\newcommand\ch{\operatorname{ch}}

\newcommand\uni{_{\operatorname{uni}}}

\newcommand\id{\operatorname{id}}

\newcommand\lp{\operatorname{\!\langle\!}}
\newcommand\rp{\operatorname{\!\rangle}}

\newcommand\Diag{\operatorname{Diag}}

\newcommand{\isom}{\,\raise2pt\hbox{$\underrightarrow{\sim}$}\,}
\numberwithin{equation}{section}

\newtheorem{thm}{Theorem}[section]
\newtheorem{lem}[thm]{Lemma}

\newtheorem{prop}[thm]{Proposition}

\def \para#1{\par\medskip\textbf{#1}
              \addtocounter{thm}{1}}

\def \remark#1{\par\medskip\noindent
                \textbf{Remark #1}
                \addtocounter{thm}{1}}

\def \remarks#1{\par\medskip\noindent
                \textbf{Remarks #1}
                \addtocounter{thm}{1}}

\begin{document}
\setlength{\baselineskip}{4.9mm}
\setlength{\abovedisplayskip}{4.5mm}
\setlength{\belowdisplayskip}{4.5mm}

%%%
%%%
\renewcommand{\theenumi}{\roman{enumi}}
\renewcommand{\labelenumi}{(\theenumi)}
\renewcommand{\thefootnote}{\fnsymbol{footnote}}
%%%
\renewcommand{\thefootnote}{\fnsymbol{footnote}}
%%%
\allowdisplaybreaks[2]
%\NoBlackBoxes
\parindent=20pt
%%%%%%%%%%%%%%%%%%%%
%%%%%%%%%%%%%%%%%%%%%%%%%%%%%%%%%%%
\medskip
\begin{center}
 {\bf Generalized Green functions \\ 
 associated to complex reflection groups }
\\
\vspace{1cm}
Toshiaki Shoji
\\
\vspace{0.7cm}
{\it To Michel Brou\'e }
\title{}
\end{center}

\begin{abstract}
In this paper, we consider the set of 
$r$-symbols in a full generality.  We construct
Hall-Littlewood functions and Kostka functions 
associated to those $r$-symbols.  We also discuss 
a multi-parameter version of those functions.  
We show that there exists a general 
algorithm of computing the multi-parameter 
Kostka functions. 
As an application, we show  that the generalized Green 
functions of symplectic groups can be described combinatorially
in terms of our (one-parameter) Kostka functions. 

\end{abstract}

\maketitle
\markboth{SHOJI}{GREEN FUNCTIONS}
\pagestyle{myheadings}

\begin{center}
{\sc Introduction}
\end{center}

\par\medskip
Lusztig introduced in [L1] the notion of symbols in order to 
parametrize unipotent representations of finite classical groups. 
By using some modification of such symbols, he 
described in [L2] the generalized Springer correspondence for classical groups.
The notion of $r$-symbols given in [Mal], [S2, 3] is a natural generalization 
of those symbols, where the original symbols correspond to 2-symbols. 
In [S2, 3], Hall-Littlewood functions and Kostka functions associated to 
$r$-symbols (namely, associated to complex reflection groups) 
were introduced. 
Actually, those $r$-symbols correspond to a subset of Lusztig's symbols, which 
is related to the Springer correspondence, and in [S2, 3], Green functions 
associated to complex reflection groups were discussed.  
In this paper, we extend the notion of $r$-symbols so that it covers  
the full set of symbols in the case where $r = 2$. 
We define Hall-Littlewood functions and Kostka functions associated to 
those $r$-symbols, which makes it possible to define generalized Green 
functions associated to complex reflection groups.  In [S2, 3], an algorithm 
of computing Kostka functions was given.
\par
In [S4], the notion of limit symbols was introduced. It is a special case of 
$r$-symbols, and in fact, in that case the description of Kostka functions 
becomes drastically simple.  
It is known that those Kostka functions (in the case where $r = 2$) are
closely related to the Springer correspondence of exotic symmetric spaces
([AH], [K], [SS]). 
In [S6], a multi-parameter version of Kostka functions (associated to limit symbols) 
was introduced.  However, the algorithm of computing those Kostka functions were not given 
there (a straightforward generalization does not work).  
In this paper, we also consider the multi-parameter version of Kostka functions 
associated to $r$-symbols of general type.       
We show that there exists a general algorithm of computing those Kostka functions. 
Although the expression is complicated in the multi-parameter case, 
it gives a simple algorithm in the one-parameter case.  
By making use of it, we show that generalized Green functions of symplectic groups 
can be described combinatorially by our Kostka functions. 
In the case of special orthogonal groups,
a combinatorial description of generalized Green functions 
is also possible by using a generalization of 
Kostka functions defined in [S3].  We will discuss those subjects elsewhere.      

\section{$r$-symbols}

\para{1.1.}
The notion of $r$-symbols was introduced in [Mal], and 
was discussed in [S2, 3]. However, in those papers, some special 
cases of $r$-symbols were treated.  In this paper, we consider 
the $r$-symbols in a full generality. 
\par
For an integer $n \ge 0, r \ge 1$, we denote by $\SP_{n,r}$ the 
set of $r$-partitions $\Bla = (\la^{(1)}, \dots, \la^{(r)})$ such that 
$\sum_{k = 1}^r|\la^{(k)}| = n$, where $|\la|$ denotes the size of 
a partition $\la$.  If $r = 1$, we simply denote it by $\SP_n$.  
For a partition $\la$, we denote by $l(\la)$ the number of non-zero parts 
of $\la$. 
\par
We fix a sequence $\Bs = (s_1, \dots, s_r)$ of non-negative integers, 
and an integer $e$ such that $s_k \le e$ for any $k$.
For $\Bm = (m_1, \dots, m_r)\in \BZ_{\ge 0}^r$, 
we denote by $\SM(\Bm)$ the set of matrices 
$A = (a^{(k)}_j)_{1 \le k \le r, 1 \le j \le m_k}$ such that $a^{(k)}_j \in \BZ_{\ge 0}$. 
We define $\wt Z^{e, \Bs}(\Bm)$ as the set of 
$\vL = (A^{(1)}, \dots, A^{(r)}) \in \SM(\Bm)$, where 
$A^{(k)} = (a_1^{(k)}, \dots, a_{m_k}^{(k)})$ is a sequence of 
non-negative integers satisfying the following properties;
\begin{equation*}
\tag{1.1.1}
\begin{aligned}
&a_i^{(k)} - a^{(k)}_{i+1} \ge e &\quad &(1 \le k \le r, 1 \le i \le m_k - 1),  \\
&a^{(k)}_{m_k} \ge s_k &\quad &(1 \le k \le r).
\end{aligned} 
\end{equation*} 
Put $\Bm' = (m_1 + 1, \dots, m_r + 1)$. We define a shift operation
$\s : \wt Z^{e,\Bs}(\Bm) \to \wt Z^{e,\Bs}(\Bm')$ as follows; 
for $\vL = (A^{(1)}, \dots, A^{(r)}) \in \wt Z^{e,\Bs}(\Bm)$, 
define $B^{(k)} = (b_1^{(k)}, \dots, b_{m_k +1}^{(k)})$ by
\begin{equation*}
b_i^{(k)} = \begin{cases}
              a_i^{(k)} + e &\quad\text{ if } i \le m_k, \\
              s_k           &\quad\text{ if } i = m_k + 1.
            \end{cases}
\end{equation*} 
If we put $\vL' = (B^{(1)}, \dots,B^{(r)})$, then $\vL' \in \wt Z^{e,\Bs}(\Bm')$.
We define a map $\s : \wt Z^{e,\Bs}(\Bm) \to \wt Z^{e,\Bs}(\Bm')$ by 
$\vL \mapsto \vL'$.
If we fix $\Bm$ and by applying the shift operation repeatedly, we obtain 
an infinite set $\bigcup_{\Bm'}\wt Z^{e,\Bs}(\Bm')$ on which $\s$ acts.
We denote by $Z^{e,\Bs}_{\Bd}$ the set of equivalence classes under $\s$ in 
this set.  Here $\Bd$ denotes the equivalence class of $\Bm \in \BZ^r_{\ge 0}$
under the operation $(a_1, \dots, a_r) \mapsto (a_1 + 1, \dots, a_r+1)$. 
A representative of $\Bd$ is given as follows; let $m$ be the smallest integer among 
$m_1, \dots, m_r$, and put $d_k = m_k - m$ for each $k$.  Then the sequence 
$(d_1, \dots, d_r)$ determines the class $\Bd$. We denote it as 
$\Bd = (d_1, \dots, d_r)$.      
The elements in $Z^{e,\Bs}_{\Bd}$ are called $r$-symbols of defect $\Bd$. 
\par
For $\Bm = (m_1, \dots, m_r) \in \BZ_{\ge 0}^r$, 
put  
$\vL^0(\Bm) = (\vL^{(1)}, \dots, \vL^{(r)})$, where 
\begin{equation*}
\tag{1.1.2}
\vL^{(k)} : s_k + (m_k -1)e \ge s_k + (m_k-2)e \ge \cdots \ge s_k + e \ge s_k.
\end{equation*}  
%%%%
%%%%
\para{1.2.}
We consider $\Bm^{\bullet} = (m_1, \dots, m_r)$, where 
\begin{equation*}
\tag{1.2.1}
m_1 = \cdots = m_{\a} = m + 1, m_{\a + 1} = \cdots = m_r = m  
\end{equation*}
for some $0 \le \a < r$ and $m \ge 0$, so that  
the defect of $\Bm^{\bullet}$ is given by 
$\Bd^{\bullet} = (1, \dots, 1, 0, \dots, 0)$ (1 : $\a$-times). 
Let $\Bm' = (m_1', \dots, m_r')$ be such that $d'_k = m_k' - m_k \ge 0$ 
for any $k$ and $\min_k\{ d'_k\} = 0$. 
Take $\vL = (a_j^{(k)}) \in \wt Z^{e,\Bs}(\Bm^{\bullet})$.  
Put, for $k = 1, \dots, r$, 
\begin{equation*}
B^{(k)} : a^{(k)}_1 + d'_ke\ge a^{(k)}_2 + d'_ke \ge \cdots \ge a^{(k)}_{m_k} + d'_ke
           \ge s_k+ (d'_k - 1)e \ge \cdots \ge s_k + e \ge s_k,
\end{equation*} 
and $\vL' = (B^{(1)}, \dots, B^{(r)})$.  Then $\vL' \in \wt Z^{e,\Bs}(\Bm')$.
The map $\vL \mapsto \vL'; \wt Z^{e,\Bs}(\Bm^{\bullet}) \to \wt Z^{e,\Bs}(\Bm')$ 
is compatible with the shift operation, and we obtain a map 
$\vf : Z^{e,\Bs}_{\Bd^{\bullet}} \to Z^{e,\Bs}_{\Bd'}$, where $\Bd'$ is the defect 
of the class containing $\Bm'$.   
Conversely, for a given $\Bm' = (m_1', \dots, m_r')$, we choose 
$\Bm^{\bullet} = (m_1, \dots, m_r)$ such that $d_k'' = m_k - m'_k \ge 0$ for any $k$, 
with $\min_k\{ d''_k\} = 0$. 
Take $\vL' \in \wt Z^{e,\Bs}(\Bm') = (b^{(k)}_j) \in \wt Z^{e,\Bs}(\Bm')$.
Put, for $k = 1, \dots, r$,
\begin{equation*}
A^{(k)} : b^{(k)}_1 + d''_ke\ge b^{(k)}_2 + d''_ke \ge \cdots \ge b^{(k)}_{m'_k} + d''_ke
           \ge s_k+ (d''_k - 1)e \ge \cdots \ge s_k + e \ge s_k, 
\end{equation*} 
and put $\vL = (A^{(1)}, \dots, A^{(r)})$.  Then $\vL \in \wt Z^{e,\Bs}(\Bm^{\bullet})$, 
and the map $\vL' \mapsto \vL$ induces a map 
$\psi : Z^{e,\Bs}_{\Bd'} \to Z^{e,\Bs}_{\Bd^{\bullet}}$. 
Since $\vf$ and $\psi$ are inverse each other, we have a bijection 
$\vf : Z^{e,\Bs}_{\Bd^{\bullet}} \isom Z^{e,\Bs}_{\Bd'}$. 
\par
For each integer $n \ge 0$, 
we choose $m$ such that $m \ge n$, and 
define $\wt Z^{e,\Bs}_{n}(\Bm^{\bullet})$ as the 
subset of $\wt Z^{e,\Bs}(\Bm^{\bullet})$ consisting of $\vL$ such that 
$\vL = \vL^0(\Bm^{\bullet}) + \Bla$, where $\Bla \in \SP_{n,r}$ with 
$\Bla \in \SM(\Bm^{\bullet})$.   
$\wt Z^{e,\Bs}_n(\Bm^{\bullet})$ is compatible with the shift operation, and 
we can define a subset $Z^{e,\Bs}_{n, \Bd^{\bullet}}$ of $Z^{e,\Bs}_{\Bd^{\bullet}}$. 
Let $\Bm' = (m_1', \dots, m_r') \in \BZ_{\ge 0}^r$ be such that 
$\sum_im_i' \equiv \a \pmod r$.  We assume that $\Bm'$ belongs to the defect $\Bd'$. 
Then, by replacing $\Bm'$ by its shift, we may assume that 
$d_k' = m_k' - m_k \ge 0$ with $\min_k\{ d'_k\} = 0$.  Put  
\begin{equation*}
\tag{1.2.2}
f(\Bd) = \sum_{k = 1}^{\a} d'_k(d'_k - 1)/2 + \sum_{k = \a+1}^rd'_k(d'_k + 1)/2.
\end{equation*}
Using the bijection $\vf : Z^{e,\Bs}_{\Bd^{\bullet}} \isom Z^{e,\Bs}_{\Bd}$, 
we define a subset $Z^{e,\Bs}_{n,\Bd}$ of $Z^{e,\Bs}_{\Bd}$ as 
$Z^{e,\Bs}_{n,\Bd} = \vf(Z^{e,\Bs}_{n',\Bd^{\bullet}})$, where $n = n' + f(\Bd)$. 
The elements in $Z^{e,\Bs}_{n,\Bd}$ are called symbols of rank $n$ and defect $\Bd$. 
It follows from the construction, the map $\Bla \mapsto \vf(\vL^0(\Bm^{\bullet}) + \Bla)$ 
gives a bijection $\SP_{n',r} \simeq Z^{e,\Bs}_{n', \Bd^{\bullet}} \simeq Z^{e,\Bs}_{n,\Bd}$.
Thus we have 
\par\medskip\noindent
(1.2.3) \ There is a bijection $\SP_{n',r} \isom Z^{e,\Bs}_{n,\Bd}$, 
where $n = n' + f(\Bd)$. 
We denote this bijection by $\Bla \mapsto \vL(\Bla)$. 
\par\medskip
For $\Bla = (\la^{(1)}, \dots, \la^{(r)}) \in \SP_{n',r}$ as above,
we fix an injective map $\th : \SP_{n',r} \to \SP_{n,r}$.  For example,  
we write $\la^{(k)} = (\la^{(k)}_1, \dots, \la^{(k)}_{m'_k})$ for each $k$, 
and define a partition $\mu^{(k)}$ by 
\begin{equation*}
\tag{1.2.4}
\mu^{(k)} : \la^{(k)}_1 + (d'_k + \d -1) \ge \la^{(k)}_2 + (d'_k + \d -2) \ge 
              \cdots \ge \la^{(k)}_{m'_k},
\end{equation*}
where $\d = 0$ for $1 \le k \le a$ and $\d = 1$ for $a + 1 \le k \le r$.
Then $\Bmu = (\mu^{(1)}, \dots, \mu^{(r)}) \in \SP_{n,r}$.  We define 
a map $\th : \SP_{n',r} \to \SP_{n,r}$ by $\Bla \mapsto \Bmu$ with $\Bmu \in \SM(\Bm')$. 
\par
Note that for a fixed $n$, the defect $\Bd$ such that 
$Z^{e,\Bs}_{n,\Bd} \ne \emptyset$ is only finite by (1.2.3). 
Let $D$ be a subset of the set of all defects $\Bd$ 
such that $Z^{e,\Bs}_{n, \Bd} \ne \emptyset$.
Put 
\begin{equation*}
\tag{1.2.5}
Z^{e,\Bs}_{n, D} = \coprod_{\Bd \in D}Z^{e,\Bs}_{n, \Bd}.
\end{equation*}
Hereafter, by fixing $e, \Bs$ (and $a$ in (1.2.1)), we simply write 
$Z^{e,\Bs}_{n, D}$, $Z^{e,\Bs}_{n, \Bd}$, etc.,  as $Z_{n,D}$, $Z_{n,\Bd}$, etc.
if there is no fear of confusion.
For a symbol $\vL \in Z_{n,D}$, we denote by $d(\vL) \in D$ the defect of $\vL$. 

\para{1.3.}
Two elements $\vL, \vL' \in Z_{n,D}$ are called similar if for 
suitable representatives $\vL \in \wt Z_n(\Bm)$, $\vL' \in \wt Z_n(\Bm')$, 
all the components contained in $\vL$ and $\vL'$ coincide each other 
with multiplicities.  This gives an equivalence relation on $Z_{n,D}$, and 
an equivalence class in $Z_{n,D}$ is called a similarity class. 
\par
For $\Bm = (m_1, \dots, m_r)$ such that $\sum_im_i \equiv \a \pmod r$, 
we define an $a$-function $a : \wt Z(\Bm) \to \BZ$ by 
\begin{equation*}
\tag{1.3.1}
a(\vL) = \sum_{a, a' \in \vL}\min (a,a') 
          - \sum_{b,b' \in \vL^0(\Bm^{\bullet})}\min(b,b'),
\end{equation*}
where $\Bm^{\bullet} = (m+1, \dots, m)$ is the unique element such that 
$\sum_im_i = \a + rm$. 
(Note that this $\Bm^{\bullet}$ is different from $\Bm^{\bullet}$ used
to define the rank of $\vL$ in 1.2.)
If we write $\vL = (a^{(k)}_j) \in \SM(\Bm)$,  
we have $s_k \le a^{(k')}_j + e$ for any $k,k',j$ by our assumption that 
$s_k \le e$. 
This implies that $a(\vL)$ is invariant under the shift operation, and 
one can define $a(\vL)$ for $\vL \in Z_{n,D}$. 
It is clear that $a$-function is constant on each similarity class in $Z_{n,D}$. 

%%%%
%%%%
\par\bigskip
\section{Symmetric functions associated to $r$-partitions}

\para{2.1.}
We fix $\Bm = (m_1, \dots, m_r)$, and consider finitely many variables
$x_{\Bm} = \{ x^{(k)}_j \mid 1 \le k \le r, 1 \le j \le m_k\}$.  We denote by 
$\Xi_{\Bm}$ the ring of symmetric polynomials
\begin{equation*}
\Xi_{\Bm} = \bigotimes_{k = 1}^r\BZ[x_1^{(k)}, \dots, x_{m_k}^{(k)}]^{S_{m_k}}.
\end{equation*}
Let $\Xi_{\Bm} = \bigoplus_{n \ge 0}\Xi^n_{\Bm}$ be the decomposition of 
$\Xi_{\Bm}$ into homogeneous parts. 
Put $\Bm' = (m_1', \dots, m_r')$ with $m_k' = m_k +1$, and define a linear map 
$\r^n_{\Bm,\Bm'} : \Xi^n_{\Bm'} \to \Xi^n_{\Bm}$ by 
$x^{(k)}_{m_k +1} = 0$, and by $x^{(k)}_j \mapsto x^{(k)}_j$ for $1 \le j \le m_k$.  
The operation $\r^n_{\Bm,\Bm'}$ is called a shift operation on the symmetric 
polynomials. 
As the inverse limit of the $\BZ$-module 
$\Xi_{\Bm}$ with respect to the shift operation $\r^n_{\Bm, \Bm'}$, we obtain 
\begin{equation*}
\wh\Xi^{n, \Bd} = \lim_{\substack{\longleftarrow  \\ \Bm }}\Xi^n_{\Bm},
\end{equation*} 
where $\Bd$ is the defect determined from $\Bm$ by the shift operation (see 1.1). 
The elements 
in $\wh\Xi^{n,\Bd}$ are called symmetric functions of degree $n$ with defect $\Bd$. 
\par
For a partition $\la$ of $n$ with $l(\la) \le m$, we can define 
a Schur polynomial $s_{\la}(y_1, \dots, y_m)$
and a monomial symmetric polynomial $m_{\la}(y_1, \dots, y_m)$.  
We assume that $\Bm = (m_1, \dots, m_r)$ with $m_k \ge n$ for any $k$. 
For $\Bla = (\la^{(1)}, \dots, \la^{(r)}) \in \SP_{n,r}$, we define 
$s_{\Bla}(x_{\Bm}), m_{\Bla}(x_{\Bm}) \in \Xi^n_{\Bm}$ by 
\begin{equation*}
s_{\Bla}(x_{\Bm}) = \prod_{k = 1}^rs_{\la^{(k)}}(x_1^{(k)}, \dots, x_{m_k}^{(k)}), 
\qquad
m_{\Bla}(x_{\Bm}) = \prod_{k=1}^rm_{\la^{(k)}}(x_1^{(k)}, \dots, x_{m_k}^{(k)}).
\end{equation*}
Then $s_{\Bla}(x_{\Bm}), m_{\Bla}(x_{\Bm}) \in \Xi^n_{\Bm}$ 
are compatible with the shift operation, and one can define 
the symmetric functions $s_{\Bla}(x), m_{\Bla}(x) \in \wh\Xi^{n,\Bd}$.  
As in the classical case, the sets $\{ s_{\Bla} \mid \Bla \in \SP_{n,r}\}, 
\{ m_{\Bla} \mid \Bla \in \SP_{n,r}\}$ give the $\BZ$-basis of $\wh\Xi^{n, \Bd}$. 
In particular, $\wh\Xi^{n, \Bd} \simeq \wh\Xi^{n, \bold 0}$ as $\BZ$-modules, where 
$\wh\Xi^{n,\bold 0}$ is the symmetric functions with defect zero. 

\para{2.2.}
Next we define power sum symmetric functions $p_{\Bla}(x)$. 
Note that the notation here is modified from the notation in [S2].
We fix a primitive $r$-th root $\z$ of unity in $\BC$. 
Consider the variables $x_{\Bm}$ for $\Bm = (m_1, \dots, m_r)$.
We define, for each integer $s\ge 1$ and $1 \le k \le r$, 
a symmetric polynomial $p_s^{(k)}(x_{\Bm})$ by 
\begin{equation*}
\tag{2.2.1}
p^{(k)}_s(x_{\Bm}) = 
     \sum_{j = 1}^r\z^{(k-1)(j-1)}p_{s}(x^{(j)}_1, \dots, x^{(j)}_{m_j}),
\end{equation*}
where $p_s(y_1, \dots, y_m) = \sum_i y_i^s$ is the classical 
power sum symmetric polynomial.  Put $p_s(x_{\Bm}) = 1$ for $s = 0$. 
For each $\Bla \in \SP_{n,r}$, we define a polynomial $p_{\Bla}(x_{\Bm})$ by 
\begin{equation*}
p_{\Bla}(x_{\Bm}) = \prod_{k = 1}^r\prod_{j = 1}^{m_k}
             p^{(k)}_{\la^{(k)}_j}(x_{\Bm}).
\end{equation*}   
Then $p_{\Bla} \in \Xi^n_{\Bm} \otimes_{\BZ}\BC$, compatible with the shift operation, 
and one can define a symmetric function $p_{\Bla}(x)$.  
$\{ p_{\Bla} \mid \Bla \in \SP_{n,r}\}$ gives a $\BC$-basis of 
$\wh\Xi^{n,\Bd}\otimes_{\BZ}\BC$. 

\para{2.3.}
Following [S6], we introduce $r$ parameters $t_1, \dots, t_r$, 
and consider the symmetric functions 
with parameters $\Bt = (t_1, \dots, t_r)$. 
Put $\BZ[\Bt] = \BZ[t_1, \dots, t_r]$, the polynomial ring with 
variables $t_1, \dots, t_r$, and $\BQ(\Bt) = \BQ(t_1, \dots, t_r)$ 
the field of rational functions with variables $t_1, \dots, t_r$. 
We denote 
$\Xi^n_{\Bm}\otimes_{\BZ}\BZ[\Bt], \wh\Xi^{n, \Bd}\otimes_{\BZ}\BZ[\Bt]$, etc., 
by $\Xi^n_{\Bm}[\Bt], \wh\Xi^{n, \Bd}[\Bt]$, etc..   
Also put $\wh\Xi^{n, \Bd}\otimes_{\BZ}\BQ(\Bt) = \wh\Xi^{n,\Bd}_{\BQ}(\Bt)$ and
$\wh\Xi^{n,\Bd}\otimes_{\BZ}\BC(\Bt) = \wh\Xi^{n,\Bd}_{\BC}(\Bt)$. 
We consider the one-parameter case also, which we denote by 
$\Xi^n_{\Bm}[t]$, etc. for the parameter $t$.  
We use the notation $\Bt_0 = (t, \dots, t)$, so that the specialization
$\Bt \mapsto \Bt_0$ gives $\Xi^n_{\Bm}[\Bt_0] = \Xi^n_{\Bm}[t]$. 
\par
In the rest of this section, we assume that $\Bd = \bold 0$, namely
$\Bm = (m, \dots, m)$ for some $m > 0$. 
\par
Following [S2], we introduce polynomials (with one-parameter $t$) 
$q_{s, \pm}^{(k)}(x_{\Bm};t) \in \Xi^s_{\Bm}[t]$.
(But note that the definition here is modified from  (2.2.1) in [S2], 
based on  Remark 5.7 in [S3].  Also note that since we assume 
that $\Bd = \bold 0$, the discussion is simplified from there). 
For any integer $s \ge 1$, we define (depending on the signature 
``$+$'' or ``$-$''), 
\begin{equation*}
\tag{2.3.1}
q_{s,\pm}^{(k)}(x_{\Bm};t) = \sum_{i = 1}^m(x_i^{(k)})^{s-1}
            \frac{\prod_{j \ge 1}(x_i^{(k)} - tx_j^{(k \mp 1)})}
                 {\prod_{j \ne i}(x_i^{(k)} - x_j^{(k)})},
\end{equation*}
and put $q_{s, \pm}^{(k)}(x_{\Bm};t) = 1$ if $s = 0$.  Note in the 
formula (2.3.1), we regard $k \mp 1$ as an element in $\BZ/r\BZ$.
Let $u$ be another variable. By using Lagrange's interpolation, 
we have a formula (see [S2, Lemma 2.3])
\begin{equation*}
\prod_{i = 1}^m\frac{1 - tux_i^{(k \mp 1)}}{1 - ux_i^{(k)}}
        = 1 + \sum_{i=1}^m\frac{ux_i^{(k)} - tux_i^{(k \mp 1)}}
              {1 - ux_i^{(k)}}\prod_{j \ne i}
               \frac{x_i^{(k)} - tx_j^{(k \mp 1)}}{x_i^{(k)} - x_j^{(k)}}. 
\end{equation*} 
It follows from this that
\begin{equation*}
\tag{2.3.2}
\sum_{s = 0}^{\infty}q^{(k)}_{s,\pm}(x_{\Bm};t)u^s = \prod_{i = 1}^m
                 \frac{1 - tux_i^{(k \mp 1)}}{1 - ux_i^{(k)}}.
\end{equation*}
In particular, we see that $q^{(k)}_{s, \pm}(x_{\Bm};t) \in \Xi^{s}_{\Bm}[t]$. 
One can check that $q^{(k)}_{s,\pm}(x_{\Bm};t)$ is compatible with the shift 
operation. 
\par
Following [S6], we define, for $\Bla \in \SP_{n,r}$, 
a polynomial $q^{\pm}_{\Bla}(x_{\Bm};\Bt)$
by 
\begin{equation*}
\tag{2.3.3}
q^{\pm}_{\Bla}(x_{\Bm};\Bt) = \prod_{k \in \BZ/r\BZ}\prod_{i = 1}^m
                             q^{(k)}_{\la^{(k)}_i, \pm}(x_{\Bm};t_{k-c}),   
\end{equation*} 
where $c = 1$ for the ``$+$''-case, and $c = 0$ for the ``$-$''-case. 
We see that $q^{\pm}_{\Bla} \in \Xi^n_{\Bm}[\Bt]$, and $q^{\pm}_{\Bla}$ 
is compatible with the shift operation.  Hence we obtain symmetric functions 
$q^{\pm}_{\Bla}(x;\Bt) \in \Xi^{n}[\Bt]$. 
Note that $q^{\pm}_{\Bla}(x;\Bt_0)$ coincides with $q^{\pm}_{\Bla}(x;t)$ 
introduced in [S2, 2.2]. 

\para{2.4.}
We prepare two types of infinitely many variables 
$x = (x^{(1)}, \dots, x^{(r)})$ and $y = (y^{(1)}, \dots, y^{(r)})$, and 
consider the product
\begin{equation*}
\tag{2.4.1}
\Om(x,y;\Bt) = \prod_{k \in \BZ/r\BZ}\prod_{i, j}
             \frac{1 - t_kx^{(k)}_iy^{(k + 1)}_j}{1 - x_i^{(k)}y_j^{(k)}}.
\end{equation*} 
(Note that this formula was introduced in [S6]. 
A one-parameter version appears in [S2, Prop. 2.5],  but the present version is modified 
from there, following Remark 5.7 in [S3]). 
The following formula appears in [S6, Prop. 2.8], which is a multi-parameter 
version of [S2, Prop. 2.5.], and is proved similarly.  

\begin{prop} %%%%  Prop. 2.5.
Under the notation above, the following formulas hold. 
\begin{align*}
\tag{2.5.1}
\Om(x,y;\Bt) &= \sum_{\Bla}q^+_{\Bla}(x;\Bt)m_{\Bla}(y) 
              = \sum_{\Bla}m_{\Bla}(x)q^-_{\Bla}(y;\Bt), \\
\end{align*} 
where $\Bla$ runs over all the $r$-partitions. 
\end{prop}

\para{2.6.}
In the one-parameter case, $\Om(x,y;t)$ is also expressed by 
using power sum symmetric functions as in [S2, (2.5.2)].  We shall show,  
in the multi-parameter case, still $\Om(x,y,\Bt)$ can be expressed 
by power sum symmetric functions.  
\par
If we denote by $p_m^{[k]}(x)$ the power sum symmetric 
function $\sum_i(x_i^{(k)})^m$ with respect to the variables $x^{(k)}$, 
(2.2.1) implies that 
\begin{equation*}
p^{(a)}_m(x) = \sum_{1 \le k \le r}\z^{(a-1)(k-1)}p^{[k]}_m(x).
\end{equation*}
We consider a matrix $\vD = (\z^{(a-1)(k-1)})_{1\le a, k \le r}$ with 
$\det \vD = \prod_{0 \le j < i < r}(\z^i - \z^j) \ne 0$. 
Let $S(\Bt) = \Diag(t_1, \dots, t_r)$ be a diagonal matrix of degree $r$, and 
put
\begin{equation*}
\tag{2.6.1}
\vD (\Bt) = \vD S(\Bt)\vD\iv.
\end{equation*}
We write $\vD(\Bt) = (\vD_{a,a'}(\Bt))_{1 \le a,a'\le r}$, where 
$\vD_{a,a'}(\Bt) \in  \BC(\Bt)$. 
For any integer $m \ge 0$, put $\Bt^m = (t_1^m, \dots, t_r^m)$. 
Then we have
\begin{equation*}
\tag{2.6.2}
\sum_{k}t_{k}^m \z^{(a-1)(k-1)}p_m^{[k]}(x) = 
        \sum_{1 \le a' \le r}\vD_{a, a'}(\Bt^m)p_m^{(a')}(x). 
\end{equation*}

\par
Let $\SP_{n,r}^2$ be the set of 
$\Bxi = (\xi^{(i,j)})_{1 \le i,j \le r}$, where $\xi^{(i,j)}$ are partitions
such that $\sum_{i,j}|\xi^{(i,j)}| = n$.
For $\Bxi \in \SP_{n,r}^2$, we define 
$\Bxi' = (\xi'^{(1)}, \dots, \xi'^{(r)}) \in \SP_{n,r}$ by 
$\xi'^{(k)} = \bigcup_{1 \le j \le r}\xi^{(k,j)}$, 
and define $\Bxi'' = (\xi''^{(1)}, \dots, \xi''^{(r)}) \in \SP_{n,r}$
by $\xi''^{(k)} = \bigcup_{1 \le i \le r}\xi^{(i,k)}$.
(Here for partitions $\la, \mu$, we denote by $\la\cup \mu$ the partition 
of $|\la| + |\mu|$ obtained by rearranging the parts of $\la$ and $\mu$ 
in a decreasing order). 
For $\Bxi \in \SP_{n,r}^2$, we define $z_{\Bxi}(\Bt)$ by 
\begin{equation*}
\tag{2.6.3}
z_{\Bxi}(\Bt) = z_{\Bxi}\iv\prod_{1 \le a, a' \le r}z_{\xi^{(a,a')}}(\Bt),
\end{equation*}
where, for $\xi^{(a,a')} = (\xi_1, \xi_2, \dots, \xi_{\ell})$
with $\ell = l(\xi^{(a,a')})$, 
\begin{equation*}
\tag{2.6.4}
z_{\xi^{(a,a')}}(\Bt) = \begin{cases} 
     \prod_{j = 1}^{\ell}
             (\d_{a,a'} - \z^{a-1}\vD_{a,a'}(\Bt^{\xi_j})) 
                    &\quad\text{ if } \ell > 0, \\
             1      &\quad\text{ if } \ell = 0.
                         \end{cases}
\end{equation*}
Moreover $z_{\Bxi}$ is defined as follows; 
for a partition 
$\la = (1^{m_1}, 2^{m_2}, \dots)$, put $z_{\la} = \prod_{i \ge 1}i^{m_i}m_i!$.
Then $z_{\Bxi} = r^{l(\Bxi)}\prod_{1 \le a, a' \le r}z_{\xi^{(a,a')}}$, 
where $l(\Bxi) = \sum_{a,a'}l(\xi^{(a,a')})$. 

\remark{2.7.}
$z_{\xi^{(a,a')}}(\Bt_0)$ is zero unless $a = a'$, hence 
$z_{\Bxi}(\Bt_0) = 0$ unless $\Bxi = (\xi^{(a,a')})$ with 
$\xi^{(a,a')} = \emptyset$ for any pair $a \ne a'$.  
Thus if $z_{\Bxi}(\Bt_0) \ne 0$, then $\Bxi$ is identified with 
an $r$-partition $\Bla = (\la^{(1)}, \dots, \la^{(r)})$ with 
$\la^{(k)} = \xi^{(k,k)}$ for each $k$.  
In this case, $z_{\Bxi}(\Bt_0)$ coincides with $z_{\Bla}(t)\iv$ 
defined in [S2, 2.4], where 
\begin{equation*}
\tag{2.7.1}
z_{\Bla}(t) = z_{\Bla}\prod_{k = 1}^rz_{\la^{(k)}}(t)
\end{equation*}
with 
\begin{equation*}
\tag{2.7.2}
z_{\la^{(k)}}(t) = \prod_{j = 1}^{l(\la^{(k)})}(1 - \z^{k-1}t^{\la_j^{(k)}})\iv.
\end{equation*}

\begin{prop}  %%%%  Prop. 2.8
Under the notation above, we have
\begin{align*}
\tag{2.8.1}
\Om(x,y;\Bt) &= \sum_{\Bxi}z_{\Bxi}(\Bt)p_{\Bxi'}(x)\ol p_{\Bxi''}(y),
\end{align*} 
where $\Bxi$ runs over all the elements in $\SP_{n,r}^2$ for any $n \ge 0$, 
and $\ol p_{\Bxi''}(y)$ 
denotes the complex conjugate of $p_{\Bxi''}(y)$.  
\end{prop}

\begin{proof}
We can write as  
\begin{equation*}
\tag{2.8.2}
\log\Om(x,y;\Bt) = \sum_k\sum_{i,j}\sum_{m = 1}^{\infty}
           \biggl(\frac{1}{m}(x_i^{(k)}y_j^{(k)})^m - 
               \frac{t_k^m}{m}(x_i^{(k)}y_j^{(k+1)})^m\biggr).
\end{equation*}  
By using the orthogonality relations of irreducible characters of $\BZ/r\BZ$, 
\begin{equation*}
\frac{1}{r}\sum_{a = 1}^r(\z^k\z^{-k'})^a = \d_{k.k'},
\end{equation*}
the right hand side of (2.8.2) can be rewritten as 
\begin{align*}
\tag{2.8.3}
\sum_{m = 1}^{\infty}\sum_{a = 1}^r\sum_{i,j, k, k'}
       &\biggl(\frac{1}{rm}\z^{(a -1)(k-1)}\z^{-(a-1)(k'-1)}(x_i^{(k)}y_j^{(k')})^m  \\
   & - \frac{t_{k}^m}{rm}\z^{(a-1)}\z^{(a-1)(k-1)}\z^{-(a-1)(k')}(x_i^{(k)}y_j^{(k'+1)})^m\biggr).
\end{align*}
By using the definition (2.2.1) of power sum symmetric functions, we have
\begin{equation*}
\tag{2.8.4}
\sum_{k,k',i,j}\frac{1}{rm}\z^{(a-1)(k-1)}\z^{-(a-1)(k'-1)}(x_i^{(k)}y_j^{(k')})^m 
     = \frac{1}{rm}p^{(a)}_m(x)\ol{p^{(a)}_m(y)}. 
\end{equation*}
On the other hand, by using (2.6.2), we have 
\begin{align*}
\tag{2.8.5}
\sum_{k,k',i,j}\frac{t_k^m}{rm}\z^{(a-1)}
         &\z^{(a-1)(k-1)}\z^{-(a-1)(k')}(x_i^{(k)}y_j^{(k'+1)})^m   \\ 
       &= \sum_{1 \le a' \le r}\frac{\z^{a-1}}{rm}
              \vD_{a,a'}(\Bt^m)p_m^{(a')}(x)\ol{p^{(a)}_m(y)}.
\end{align*}
Thus (2.8.3) can be written as 
\begin{equation*}
\tag{2.8.6}
\sum_{1\le a, a' \le r}\sum_{m = 1}^{\infty}\frac{\d_{a,a'} - \z^{a-1}\vD_{a,a'}(\Bt^m)}{rm}
                p_m^{(a')}(x)\ol{p_m^{(a)}(y)}.
\end{equation*}
Hence we have 
\begin{align*}
\Om(x,y;\Bt) = \prod_{a = 1}^r\prod_{a' = 1}^r\prod_{m = 1}^{\infty}
   \exp\biggl(\frac{\d_{a,a'} - \z^{a-1}\vD_{a,a'}(\Bt^m)}{rm}
                p_m^{(a')}(x)\ol{p_m^{(a)}(y)}\biggr).
\end{align*}
This implies that 
\begin{align*}
\tag{2.8.7}
\Om(x,y;\Bt) = \sum_{\Bxi}z_{\Bxi}(\Bt)\prod_{a,a'}\prod_j
                p^{(a')}_{\xi_j}(x)\ol{p^{(a)}_{\xi_j}(y)},
\end{align*}
where $\Bxi$ runs over all elements in $\bigcup_{n \ge 0}\SP^2_{n,r}$, 
and  we write $\Bxi = (\xi^{(a,a')})$ with $\xi^{(a,a')} = (\xi_1, \dots, \xi_k)$. 
In view of the discussion in  2.6, the proposition follows from (2.8.7). 
\end{proof}

\para{2.9.}
For any $\Bla, \Bmu \in \SP_{n,r}$, put 
\begin{equation*}
\tag{2.9.1}
z_{\Bla,\Bmu}(\Bt) = \sum_{\substack{
                \Bxi \in \SP^2_{n,r} \\ \Bxi' = \Bla, \Bxi'' = \Bmu}}z_{\Bxi}(\Bt). 
\end{equation*}
Then (2.8.1) can be rewritten as follows.
\begin{equation*}
\tag{2.9.2}
\Om(x,y;\Bt) = \sum_{\Bla, \Bmu}z_{\Bla,\Bmu}(\Bt)p_{\Bla}(x)\ol p_{\Bmu}(y), 
\end{equation*}
where $\Bla, \Bmu$ run over all $r$-partitions. 
\par
By Remark 2.7, $z_{\Bla,\Bmu}(\Bt_0) = z_{\Bla}(t)\iv$ if $\Bla = \Bmu$, 
and is equal to zero otherwise.  Hence (2.9.2) implies that
\begin{equation*}
\tag{2.9.3}
\Om(x,y;\Bt_0) = \Om(x,y;t) = \sum_{\Bla}z_{\Bla}(t)\iv p_{\Bla}(x)\ol p_{\Bla}(y), 
\end{equation*}
which coincides with the formula in [S2, (2.5.2)].

\para{2.10.}
It follows from 2.1 that $\{ s_{\Bla} \mid \Bla \in \SP_{n,r}\}$,
$\{ m_{\Bla} \mid \Bla \in \SP_{n,r}\}$ give the $\BZ[\Bt]$-basis of 
$\BZ[\Bt]$-module $\wh\Xi^{n, \bold 0}[\Bt]$. 
On the other hand, 
it is known by [S6, Lemma 2.6] that $\{ q^{\pm}_{\Bla} \mid \Bla \in \SP_{n,r}\}$
gives a $\BQ(\Bt)$-basis of $\wh\Xi^{n, \bold 0}_{\BQ}(\Bt)$. 
We define a non-degenerate bilinear form 
$\lp\ ,\ \rp : \wh\Xi^{n,\bold 0}_{\BQ}(\Bt) \times \wh\Xi^{n,\bold 0}_{\BQ}(\Bt) \to \BQ(\Bt)$
by 
\begin{equation*}
\tag{2.10.1}
\lp q^+_{\Bla}(x;\Bt), m_{\Bmu}(x)\rp = \d_{\Bla,\Bmu}.
\end{equation*}   
By using a similar argument as in [Mac, Chap.1, 4], 
we obtain the following result as a corollary of Proposition 2.5.
\begin{equation*}
\tag{2.10.2}
\lp m_{\Bla}(x), q^-_{\Bmu}(x;\Bt) \rp = \d_{\Bla,\Bmu}. 
\end{equation*}

%%%%
%%%%
\par\bigskip
 
\section{Hall-Littlewood functions associated to $r$-symbols}

\para{3.1.}
Hall-Littlewood functions associated to $r$-symbols were constructed 
in [S2], but the discussion there was focussed on the symbols 
of fixed defect, namely the symbols in connection with 
the Springer correspondence.  In order to extend it to 
the case of arbitrary defects, we will reformulate the construction of Hall-Littlewood 
functions in [S2], based on the discussion in [S6].  

\para{3.2.}
Take $\Bm = (m_1, \dots, m_r)$, and let $\Bd = (d_1, \dots, d_r)$, 
$\Bm^{\bullet} = (m, \dots, m)$ be as in 1.1, 1.2.
For a given integer $n \ge 0$, put $n' = n - f(\Bd)$.  
Let $\th : \SP_{n',r} \to \SP_{n,r}$ be the map defined in 1.2.
By noticing that 
$\{ s_{\Bla} \mid \Bla \in \SP_{n',r}\}$ is a basis of $\wh\Xi^{n',\bold 0}[\Bt]$, 
we define a map $\vT = \vT_{\Bd} : \wh\Xi^{n',\bold 0}[\Bt] \to \wh\Xi^{n,\Bd}[\Bt]$ by 
$s_{\Bla}(x) \mapsto s_{\th(\Bla)}(x)$.  
(Here we regard $\Bla \in \SM(\Bm^{\bullet})$ as an element 
in $\SM(\Bm)$ by a suitable addition of 0's.) 
$\vT$ gives an embedding $\wh\Xi^{n',\bold 0}[\Bt] \hra \wh\Xi^{n,\Bd}[\Bt]$.  We denote by 
$\Xi^{n,\Bd}[\Bt]$ the image of $\vT$.  
We also put $\Xi^{n,\Bd}_{\BQ}(\Bt) = \Xi^{n,\Bd}[\Bt]\otimes_{\BZ[\Bt]}\BQ(\Bt)$, 
$\Xi^{n,\Bd}_{\BC}(\Bt) = \Xi^{n,\Bd}[\Bt]\otimes_{\BZ[\Bt]}\BC(\Bt)$.  
$\vT$ can be extended to the map on these spaces, which we also denote by $\vT$. 
By (1.2.3), the map $\Bla \mapsto \vL(\Bla) = \vf(\vL^0(\Bm^{\bullet}) + \Bla)$ gives a
bijective correspondence $\SP_{n',r} \isom Z_{n,\Bd}$. 
For each $\Bla \in \SP_{n',r}$, we define  
$m_{\vL(\Bla)}, q^{\pm}_{\vL(\Bla)}$ by $m_{\vL(\Bla)} = \vT(m_{\Bla}), 
q^{\pm}_{\vL(\Bla)} = \vT(q^{\pm}_{\Bla})$. Also put $s_{\vL(\Bla)} = \vT(s_{\Bla})$.
Then $\{ s_{\vL} \mid \vL \in Z_{n,\Bd}\}$, $\{ m_{\vL} \mid \vL \in Z_{n,\Bd}\}$
give $\BZ[\Bt]$-bases of $\Xi^{n,\Bd}[\Bt]$. Moreover,  
$\{ q^{\pm}_{\vL} \mid \vL \in Z_{n,\Bd}\}$ gives a $\BQ(\Bt)$-basis of 
$\Xi^{n, \Bd}_{\BQ}(\Bt)$.
On the other hand, if we put $p_{\vL(\Bla)} = \vT(p_{\Bla})$, then 
$\{ p_{\vL} \mid \vL \in Z_{n,\Bd}\}$ gives a $\BC(\Bt)$-basis of $\Xi^{n,\Bd}_{\BC}(\Bt)$. 
\par
For $\vL = \vL(\Bla), \vL' = \vL(\Bmu)$ with $\Bla, \Bmu \in \SP_{n',r}$, 
put $z_{\vL,\vL'}(\Bt) = z_{\Bla, \Bmu}(\Bt)$ (see (2.9.1)). We define 
$\Om^{n,\Bd}(x,y;\Bt)$ by 
\begin{equation*}
\tag{3.2.1}
\Om^{n,\Bd}(x,y;\Bt) = \sum_{\vL, \vL' \in Z_{n,\Bd}}z_{\vL, \vL'}(\Bt)
       p_{\vL}(x)\ol p_{\vL'}(y). 
\end{equation*}
Note that $\ol p_{\vL(\Bla)} = \vT(\ol p_{\Bla})$ since $\vT$ is defined over $\BQ(\Bt)$. 
Thus by Proposition 2.5 and (2.9.2), we have
%%%%
%%%%
\begin{prop}  %%%%  Prop. 3.3
Under the notation as above, 
\begin{equation*}
\tag{3.3.1}
\Om^{n,\Bd}(x,y;\Bt) = \sum_{\vL \in Z_{n,\Bd}}q^+_{\vL}(x;\Bt)m_{\vL}(y) 
                   = \sum_{\vL \in Z_{n,\Bd}}m_{\vL}(x)q^-_{\vL}(y;\Bt).
\end{equation*}
\end{prop}

\para{3.4.}
Put $\Xi^{n,D}_{\BQ}(\Bt) = \bigoplus_{\Bd \in D}\Xi^{n,\Bd}_{\BQ}(\Bt)$, where 
$D$ is as in 1.2. 
Then $\{ s_{\vL} \mid \vL \in Z_{n,D} \}$, $\{ m_{\vL} \mid \vL \in Z_{n,D}\}$
and $\{ q^{\pm}_{\vL} \mid \vL \in Z_{n,D}\}$ give bases of $\Xi^{n,D}_{\BQ}(\Bt)$.  
If we define $\Xi^{n,D}_{\BC}(\Bt)$ similarly, $\{ p_{\vL} \mid \vL \in Z_{n,D}\}$ gives 
a basis of $\Xi^{n,D}_{\BC}(\Bt)$. 
We define a bilinear form $\lp\ ,\ \rp$ on $\Xi^{n,\Bd}_{\BQ}(\Bt)$ by 
making use of the isomorphism 
$\vT_{\Bd} : \Xi^{n',\bold 0}_{\BQ}(\Bt) \isom \Xi^{n,\Bd}_{\BQ}(\Bt)$, then extend it
to the bilinear form on $\Xi^{n,D}_{\BQ}(\Bt)$ subject to the condition that 
$\Xi^{n,\Bd}_{\BQ}(\Bt)$ and $\Xi^{n,\Bd'}_{\BQ}(\Bt)$ are orthogonal if $\Bd \ne \Bd'$. 
Also we extend this to the sesqui linear form on $\Xi^{n,D}_{\BC}(\Bt)$. 
Then by (2.10.1) and (2.10.2), we have
\begin{align*}
\tag{3.4.1}
\lp q^+_{\vL}(x;\Bt), m_{\vL'}(x)\rp = \d_{\vL,\vL'}, \quad
\lp m_{\vL}(x), q^-_{\vL'}(x;\Bt)\rp = \d_{\vL,\vL'}. 
\end{align*}

\par
The isomorphism $\vT_{\Bd} : \wh\Xi^{n',\bold 0}_{\BQ}(\Bt) \isom \Xi^{n,\Bd}_{\BQ}(\Bt)$ 
makes 
sense for $\Bt = \bold 0$, and we can define a $\BQ$-vector space 
$\Xi^{n,\Bd}_{\BQ} = \vT_{\Bd}(\Xi^{n',\bold 0}_{\BQ})$. 
Put $\Xi^{n,D}_{\BQ} = \bigoplus_{\Bd \in D}\Xi^{n,\Bd}_{\BQ}$.  Then 
$\{ s_{\vL} \mid \vL \in Z_{n,D}\}$ gives a $\BQ$-basis of $\Xi^{n,D}_{\BQ}$. 
By applying the discussion in [Mac, 1.4] for variables $x^{(k)}, y^{(k)}$ in 
the case where $\Bt = \bold 0$, we obtain the formula
\begin{equation*}
\tag{3.4.2}
\Om(x,y; \bold 0) = \prod_{1 \le k \le r}\prod_{i,j}(1 - x^{(k)}_iy^{(k)}_j)\iv
            = \sum_{\Bla}s_{\Bla}(x)s_{\Bla}(y). 
\end{equation*}
Thus, by applying $\vT_{\Bd}$ for each $\Bd$, we have
\begin{equation*}
\tag{3.4.3}
\Om^{n,\Bd}(x,y;\bold 0) = \sum_{\vL \in Z_{n,\Bd}}s_{\vL}(x)s_{\vL}(y).
\end{equation*}
We define a symmetric bilinear form $\lp \ , \ \rp_0$ on $\Xi^{n,D}_{\BQ}$ 
by $\lp s_{\vL}, s_{\vL'}\rp_0 = \d_{\vL,\vL'}$ ($\vL,\vL' \in Z_{n,D}$).  
\par
Let $\SA$ be the $\BQ[\Bt]$-subalgebra of $\BQ(\Bt)$ generated by rational functions 
on $\Bt$ which have no pole at $\Bt = \bold 0$. 
Put $\Xi^{n,\Bd}_{\SA} = \Xi^{n,\Bd}[\Bt]\otimes_{\BQ[\Bt]}\SA$, and define 
an $\SA$-submodule $\Xi^{n,D}_{\SA}$ of $\Xi^{n,D}_{\BQ}(\Bt)$ by 
$\Xi^{n,D}_{\SA} = \bigoplus_{\Bd \in D}\Xi^{n,\Bd}_{\SA}$.  
By the specialization $\Bt \mapsto \bold 0$, we have a natural map
$\Xi^{n,D}_{\SA} \to \Xi^{n,D}_{\BQ}$. Under this specialization, the bilinear form 
$\lp\ ,\  \rp$ on $\Xi^{n,D}_{\SA}$ gives rise to a bilinear form on $\Xi^{n,D}_{\BQ}$, 
which coincides with $\lp\ ,\ \rp_0$ by (3.4.3). 
 
\para{3.5.}
We consider a total order $\vL \lve \vL'$ on $Z_{n,D}$ satisfying the following properties;
\begin{enumerate}
\item
$\vL \lv \vL'$ if $a(\vL) > a(\vL')$.
\item
Each similarity class in $Z_{n,D}$ gives rise to an interval with respect to the
total order $\lve$ on $Z_{n,D}$. 
\end{enumerate} 
\par
In the discussion below, we denote by $\vL \sim \vL'$ if $\vL, \vL' \in Z_{n,D}$ 
belong to the same similarity class. 
\par
We can define Hall-Littlewood functions by making use of the total order $\lve$ 
and the bilinear form $\lp\ , \ \rp $ on $\Xi^{n,D}_{\BQ}(\Bt)$. 

\begin{prop} %%%%  Prop 3.6
For each $\vL \in Z_{n,D}$ and a signature $\pm$, there exists a unique  
$P^{\pm}_{\vL}(x;\Bt) \in \Xi^{n,D}_{\SA}$ satisfying the following properties;
\begin{enumerate}
\item 
$P^{\pm}_{\vL}(x;\Bt)$ can be expressed as 
\begin{equation*}
\tag{3.6.1}
P^{\pm}_{\vL}(x;\Bt) = s_{\vL}(x) + \sum_{\vL' \lv \vL}u^{\pm}_{\vL,\vL'}(\Bt)s_{\vL'}(x),
\end{equation*}
where $u^{\pm}_{\vL,\vL'}(\Bt) \in \SA$, and $u^{\pm}_{\vL, \vL'}(\Bt) = 0$ if $\vL' \sim \vL$.
\item
$\lp P^+_{\vL}, P^-_{\vL'}\rp = 0$ if $\vL \not\sim \vL'$. 
\item
$P^{\pm}_{\vL}(x;\bold 0) = s_{\vL}(x)$. 
\end{enumerate}
\end{prop}

\begin{proof}
We shall construct the functions $P^{\pm}_{\vL}(x;\Bt)$ satisfying the properties 
(i) $\sim$ (iii) by induction on the total order $\lve$. 
Let $\vL_0$ be the minimum element with respect to $\lve$. 
By (i), $P^{\pm}_{\vL_0}$ must coincide with $s_{\vL_0}$. 
Moreover by (i), $P^{\pm}_{\vL} = s_{\vL}$ 
if $\vL \sim \vL_0$. These functions satisfy the conditions (i), (ii) and (iii).
Hence $P^{\pm}_{\vL}$ can be constructed for the similarity class containing $\vL_0$.
Now take $\vL \in Z_{n,D}$, and assume, for $\vL'\lv \vL, \vL'' \lv \vL$, that  
$P^{+}_{\vL'}, P^-_{\vL''}$ are constructed, where they satisfy the properties (i), (iii) and 
\par\medskip
(ii)$'$ : $\lp P^+_{\vL'}, P^-_{\vL''}\rp = 0$ for $\vL' \not\sim \vL''$. 
\medskip

We note, under our assumption, that   
the condition (i) is equivalent to the condition 
\begin{equation*}
\tag{3.6.2}
P^{\pm}_{\vL} = s_{\vL} + \sum_{\vL' \lv \vL}d^{\pm}_{\vL,\vL'}P^{\pm}_{\vL'},
\end{equation*}
where $d^{\pm}_{\vL,\vL'} \in \SA$ and $d^{\pm}_{\vL,\vL'} = 0$ if $\vL' \sim \vL$.
If we consider the scalar product with $P^-_{\vL''}$ on both sides of (3.6.2), 
we have, by (ii)$'$, 
\begin{equation*}
\tag{3.6.3}
\lp P^{+}_{\vL}, P^-_{\vL''}\rp = \lp s_{\vL}, P^-_{\vL''}\rp + 
                \sum_{\vL' \sim \vL''}d^+_{\vL,\vL'}\lp P^+_{\vL'}, P^-_{\vL''}\rp.
\end{equation*}
By (iii) and the discussion in 3.4, we see that 
$\lp P^+_{\vL'}, P^-_{\vL''}\rp|_{\Bt = \bold 0} = 
      \lp s_{\vL'}, s_{\vL''}\rp_0 = \d_{\vL',\vL''}$.
In particular, if we put $\SC$ the similarity class containing $\vL''$, 
the matrix 
$B_{\SC} = (\lp P^+_{\vL'}, P^-_{\vL''}\rp)_{\vL',\vL'' \in \SC}$ 
satisfies the property that $(\det B_{\SC})_{\Bt = \bold 0} \ne 0$. 
It follows that the system of linear equations with respect to the unknown 
variables $u_{\vL'}$ ($\vL' \in \SC$) 
\begin{equation*}
\tag{3.6.4}
\lp s_{\vL}, P^-_{\vL''}\rp + 
   \sum_{\vL' \in \SC}u_{\vL'}\lp P^+_{\vL'}, P^-_{\vL''}\rp = 0  
\end{equation*}
has a unique solution $u_{\vL'} = d^+_{\vL,\vL'}$. 
Since $(\det B_{\SC})_{\Bt = \bold 0} \ne 0$, we have $d^+_{\vL,\vL'} \in \SA$. 
We now define $d^+_{\vL,\vL'}$ as above for any similarity class $\SC$ containing 
$\vL''$ such that $\vL'' \lv \vL, \vL'' \not\sim \vL$, and define $P^+_{\vL}$ 
by using the formula (3.6.2). 
Then for any $\vL'\lv \vL, \vL' \not\sim \vL$, we have $\lp P^+_{\vL}, P^-_{\vL'}\rp = 0$. 
Also it follows from the construction that $d^+_{\vL, \vL'} = 0$ for $\vL' \sim \vL$.   
On the other hand, if we put $\Bt = \bold 0$ in (3.6.4), we have 
$d^+_{\vL,\vL'}(\bold 0) = 0$ 
since $\lp s_{\vL}, P^-_{\vL''}\rp|_{\Bt = \bold 0} = \lp s_{\vL}, s_{\vL''}\rp_0 = 0$.  
Hence $P^+_{\vL}(x;\bold 0) = s_{\vL}(x)$. 
Thus $P^+_{\vL}$ satisfies the condition (i), (iii) and (ii)$'$. 
\par
Instead of (3.6.3), if we compute $\lp P^+_{\vL'}, P^-_{\vL}\rp$, in a similar way 
as above we can construct 
$P^-_{\vL}$ satisfying (i), (iii) and (ii)$'$.
Thus $P^{\pm}_{\vL}$ is constructed.  
\par
Next we show the uniqueness. 
By induction, we may assume that $P^{\pm}_{\vL'}$ 
is already determined uniquely for any $\vL' \lv \vL$.
Suppose that both of $P^+_{\vL}, P'^+_{\vL}$ satisfy the condition (i) and (ii).  
Then by (i), 
\begin{equation*}
X = P^+_{\vL} - P'^+_{\vL} = 
     \sum_{\substack{\vL' \lv \vL \\ \vL' \not\sim \vL}}b_{\vL'}P^+_{\vL'}.
\end{equation*}
Hence $\lp X, P^-_{\vL'}\rp = 0$ for $\vL' \gv \vL$ or $\vL' \sim \vL$. 
On the other hand, for $\vL' \lv \vL, \vL' \not\sim \vL$, we have
$\lp P^+_{\vL}, P^-_{\vL'}\rp = 0$, 
$\lp P'^+_{\vL}, P^-_{\vL'}\rp = \lp P'^+_{\vL}, P'^-_{\vL'}\rp = 0$. 
Hence $\lp X, P^-_{\vL'}\rp = 0$. 
It follows that $\lp X, P^-_{\vL'}\rp = 0$ for any $\vL'$. 
Since $\{ P^-_{\vL'} \mid \vL' \in Z_{n,D}\}$ is a basis of $\Xi^{n,D}_{\BQ}(\Bt)$, 
and $\lp\ , \ \rp$ is non-degenerate, we see that $X = 0$, namely 
$P^+_{\vL} = P'^+_{\vL}$. The uniqueness of $P^-_{\vL}$ can be shown similarly. 
The proposition is proved. 
\end{proof}

\para{3.7.}
It follows from Proposition 3.6 (i), that $\{ P^{\pm}_{\vL} \mid \vL \in Z_{n,D}\}$ 
gives rise to an $\SA$-basis of $\Xi^{n,D}_{\SA}$. 
Let $\SC$ be a similarity class in $Z_{n,D}$.  The proof of Proposition 3.6 
shows that $B = (\lp P^+_{\vL}, P^-_{\vL'}\rp)_{\vL,\vL' \in \SC}$ is non-singular
and $(\det B)|_{\Bt = \bold 0} \ne 0$.  
Hence if we put $B\iv = B_{\SC} = (b_{\vL,\vL'}(\Bt))$, then $b_{\vL,\vL'}(\Bt) \in \SA$.  
We now define symmetric functions $Q^{\pm}_{\vL}(x;\Bt) \in \Xi^{n,D}_{\SA}$ by 
\begin{align*}
\tag{3.7.1}
Q^+_{\vL}(x;\Bt) &= \sum_{\vL' \sim \vL}b_{\vL,\vL'}(\Bt)P^+_{\vL'}(x;\Bt), \\
Q^-_{\vL}(x;\Bt) &= \sum_{\vL' \sim \vL}b_{\vL',\vL}(\Bt)P^-_{\vL'}(x;\Bt).
\end{align*}
Then $\{ Q^{\pm}_{\vL} \mid \vL \in Z_{n,D}\}$ gives rise to an $\SA$-basis of 
$\Xi^{n,D}_{\SA}$. $P^{\pm}_{\vL}$ and $Q^{\pm}_{\vL}$ are called Hall-Littlewood 
functions attached to $r$-symbols. 
\par
The orthogonality relations for $P^{\pm}_{\vL}$ in Proposition 3.6 (ii)
and the definition of $Q^{\pm}_{\vL}$ imply that 
\begin{equation*}
\tag{3.7.2}
\lp P^+_{\vL}, Q^-_{\vL'}\rp = \lp Q^+_{\vL}, P^-_{\vL'}\rp = \d_{\vL,\vL'}. 
\end{equation*}
\par
Here we put
\begin{equation*}
\tag{3.7.3}
\Om^{n,D}(x,y;\Bt) = \sum_{\Bd \in D}\Om^{n,\Bd}(x,y;\Bt).
\end{equation*}
Then (3.2.1) can be rewritten as 
\begin{equation*}
\tag{3.7.4}
\Om^{n,D}(x,y;\Bt) = \sum_{\vL,\vL' \in Z_{n,D}}
           z_{\vL,\vL'}(\Bt) p_{\vL}(x)\ol p_{\vL}(y),
\end{equation*}
where we put $z_{\vL,\vL'}(\Bt) = 0$ if $d(\vL) \ne d(\vL')$. 
Now the orthogonality relations (3.7.2) can be rewritten as follows: 

\begin{lem}  %%%%  Lemma 3.8.
Under the notation above, 
\begin{align*}
\tag{3.8.1}
\Om^{n,D}(x,y;\Bt) &= \sum_{\vL,\vL'}
      b_{\vL',\vL}(\Bt)P^+_{\vL}(x;\Bt)P^-_{\vL'}(y;\Bt), \\
\tag{3.8.2}
\Om^{n,D}(x,y;\Bt) &= \sum_{\vL}P^+_{\vL}(x;\Bt)Q^-_{\vL}(y;\Bt) 
              = \sum_{\vL}Q^+_{\vL}(x;\Bt)P^-_{\vL}(y;\Bt), 
\end{align*}
where $\vL,\vL'$ run over all elements in $Z_{n,D}$. We put $b_{\vL,\vL'}(t) = 0$ 
unless $\vL \sim \vL'$.
\end{lem}
\begin{proof}
Since $P^+_{\vL}$ and $Q^-_{\vL'}$, $Q^+_{\vL}$ and $P^-_{\vL'}$, are dual basis 
with respect to the form $\lp\ ,\ \rp$, (3.8.2) is obtained by a similar argument 
as in [Mac, Chap. I,4].  Then by using (3.7.1), we obtain (3.8.1).  
\end{proof}

The following result gives a characterization of Hall-Littlewood functions 
$P^{\pm}_{\vL}$ and $Q^{\pm}_{\vL}$. 

\begin{thm}  %%%%  Theorem 3.9.
$P^{\pm}_{\vL}(x;\Bt)$ is determined uniquely as the function satisfying the 
following properties. Let $\SC$ be the similarity class containing $\vL$. 
\begin{enumerate}
\item
$P^{\pm}_{\vL}(x;\Bt)$ can be written as 
\begin{equation*}
\tag{3.9.1}
P^{\pm}_{\vL}(x;\Bt) = \sum_{(*)}c_{\vL,\vL'}(\Bt)q^{\pm}_{\vL'}(x;\Bt),
\end{equation*} 
where {\rm (*)} is that $\vL'$ runs over all elements in $Z_{n,D}$ 
such that $\vL' \sim \vL$ or 
$\vL' \gve \vL$. Moreover $c_{\vL,\vL'}(\Bt) \in \BQ(\Bt)$, and 
$c(\Bt)_{\SC} = (c_{\vL',\vL''}(\Bt))_{\vL',\vL'' \in \SC}$
is a non-singular matrix. 
\item
$P^{\pm}_{\vL}(x;\Bt)$ can be written as 
\begin{equation*}
\tag{3.9.2}
P^{\pm}_{\vL}(x;\Bt) = \sum_{(**)}u_{\vL,\vL'}(\Bt)s_{\vL'}(x),
\end{equation*}
where {\rm (**)}  is that $\vL'$ runs over all elements in $Z_{n,D}$ such that 
$\vL' \sim \vL$ or $\vL' \lve \vL$.  Moreover $u_{\vL,\vL'}(\Bt) \in \BQ(\Bt)$, 
and $u(\Bt)_{\SC} = (u_{\vL',\vL''}(\Bt))_{\vL',\vL'' \in \SC}$ is the identity matrix.  
\end{enumerate}
Similarly, $Q^{\pm}_{\vL}(x;\Bt)$ is determined uniquely by similar formulas 
as in  (3.9.1) and (3.9.2), where 
the conditions for $c(\Bt)_{\SC}$ and $u(\Bt)_{\SC}$ are interchanged so that 
$c(\Bt)_{\SC}$ is the identity matrix, and  
$u(\Bt)_{\SC}$ is a non-singular matrix.   
\end{thm}

\begin{proof}
For two bases $u = \{ u_{\vL}\}, v = \{ v_{\vL} \}$ of $\Xi^{n,D}_{\BQ}(\Bt)$, we denote by
$M = M(u,v)$ the transition matrix $(m_{\vL,\vL'})$ of those two bases, namely 
$u_{\vL} = \sum_{\vL'}m_{\vL,\vL'}v_{\vL'}$. We consider the following 
bases of $\Xi^{n,D}_{\BQ}(\Bt)$, $P^{\pm} = \{ P^{\pm}_{\vL}\}$, $q^{\pm} = \{ q^{\pm}_{\vL}\}$, 
$s = \{ s_{\vL}\}$, and $m = \{ m_{\vL}\}$. 
Put
\begin{equation*}
A_{\pm} = M(q^{\pm}, P^{\pm}), \qquad B_{\pm} = M(m, P^{\pm}). 
\end{equation*}
In the following discussion, we consider the matrices indexed by $Z_{n,D}$ as 
the block matrices with respect to the partition of $Z_{n,D}$ arising from similarity classes. 
We show that $A_{\pm}$ is upper triangular as a block matrix. 
By Proposition 3.6, $M(P^{\pm}, s)$ is a lower triangular block matrix, and the diagonal 
block is the identity matrix. 
On the other hand, we note that $M(s,m)$ is a lower uni-triangular block matrix. In fact, 
for $\Bla, \Bmu \in \SP_{n',r}$, $s_{\Bla}$ is written as a linear combination of 
$m_{\Bmu}$.  Moreover, if we write $\Bla = (\la^{(1)}, \dots, \la^{(r)})$, 
$\Bmu = (\mu^{(1)}, \dots, \mu^{(r)})$, then $|\la^{(k)}| = |\mu^{(k)}|$ and 
$\mu^{(k)} \le \la^{(k)}$ for any $k$, where the coefficient of $m_{\Bla}$ equals 1.
By applying $\vT_{\Bd}$, we see that $s_{\vL}$ is a linear combination of 
$m_{\vL'}$, where if we write $\vL = (A^{(1)}, \dots, A^{(r)})$, 
$\vL' = (B^{(1)}, \dots, B^{(r)})$, then $A^{(k)}, B^{(k)}$ are partitions such that
$|A^{(k)}| = |B^{(k)}|$ and that $B^{(k)} < A^{(k)}$ unless $B^{(k)} = A^{(k)}$ 
for each $k$.  It follows,
by using [S2, Lemma 3.9], that $a(\vL') \ge a(\vL)$ and the equality holds only when 
$\vL' \sim \vL$. Hence, by our assumption 3.5, we have $\vL \lve \vL'$ with equality only when 
$\vL' \sim \vL$. This implies that $M(s,m)$ is lower uni-triangular. 
\par  
From the above discussion, we see that $B\iv_{\pm} = M(P^{\pm}, s)M(s,m)$ is lower 
uni-triangular block matrix, and so is $B_{\pm}$. 
Here put $D_{\pm} = {}^tB_{\mp}A_{\pm}$.   
If we write $A_+ = (A^+_{\vL,\vL'})$, $B_- = (B^-_{\vL,\vL'})$, we have
\begin{align*}
\tag{3.9.3}
\sum_{\vL \in Z_{n,D}}q^+_{\vL}(x;\Bt)m_{\vL}(y) 
         &= \sum_{\vL,\vL',\vL''}A^+_{\vL,\vL'}B^-_{\vL,\vL''}
                P^+_{\vL'}(x;\Bt)P^-_{\vL''}(y;\Bt) \\
         &= \sum_{\vL',\vL''}b_{\vL'',\vL'}P^+_{\vL'}(x;\Bt)P^-_{\vL''}(y;\Bt).
\end{align*}
The second equality is obtained by comparing (3.3.1) and (3.8.1). 
We compare the first and the second equalities in (3.9.3). 
Since $P^+_{\vL'}(x;\Bt)P^-_{\vL''}(y;\Bt)$ are linearly independent, 
this implies that $D_+ = {}^tB_-A_+$ is a diagonal block matrix, and that 
the diagonal block $D^+_{\SC}$ corresponding to the similarity class $\SC$ 
coincides with $B_{\SC} = (b_{\vL',\vL''})$. 
Hence $A_+$ is an upper triangular block matrix, and the diagonal block 
corresponding to $\SC$ coincides with $B_{\SC}$. 
By applying a similar argument to the second identity in (3.3.1), 
we obtain that $A_-$ is an upper triangular block matrix, and the diagonal
block corresponding to $\SC$ coincides with ${}^tB_{\SC}$. 
This shows that $P^{\pm}_{\vL}(x;\Bt)$ satisfies the relation in (i). 
The fact that $P^{\pm}_{\vL}$ satisfies the relation (ii) was already verified 
in Proposition 3.6.
\par
We show the uniqueness of $P^{\pm}_{\vL}$. 
Assume that $R(x;\Bt)$ satisfies the condition (i) and (ii) as above, 
replacing $P^{\pm}_{\vL}$ by $R$. Since $P^{\pm}_{\vL}$ satisfies (i) and (ii), 
one can write as
\begin{align*}
s_{\vL'} &= \sum_{(*)}u'_{\vL',\vL''}P^+_{\vL''}, \\
q^+_{\vL'} &= \sum_{(**)}c'_{\vL',\vL''}P^+_{\vL''},
\end{align*}
where (*) and (**) are similar conditions as (*), (**) in (3.9.1) and (3.9.2), 
by replacing $\vL$ by $\vL'$. 
By substituting these equations into the expression of $R$, we have
\begin{align*}
R = \sum_{(1)}u''_{\vL,\vL''}P^+_{\vL''} 
  = \sum_{(2)}c''_{\vL,\vL''}P^+_{\vL''}, 
\end{align*}
where in (1) $\vL''$ runs over all the elements such that 
$\vL'' \sim \vL$ or $\vL'' \lv \vL$, and in (2) $\vL''$ runs over 
all the elements such that $\vL'' \sim \vL$ or $\vL'' \gve \vL$. 
Moreover, $u''(\Bt)_{\SC} = (u''_{\vL',\vL''}(\Bt))$ is the identity matrix. 
It follows that $R = P^+_{\vL}$. The proof for $P^-_{\vL}$ is similar. 
Thus the uniqueness of $P^{\pm}_{\vL}$ is proved. 
\par
Let $c^{\pm}_{\vL,\vL'}(\Bt)$ be the coefficient $c_{\vL,\vL'}(\Bt)$ for 
$P^{\pm}_{\vL}(x;\Bt)$ in (3.9.1).  It follows from the proof of (i) in the theorem, 
that $(c^+(\Bt))_{\SC} = B\iv_{\SC}$, $(c^-(\Bt))_{\SC} = {}^tB\iv_{\SC}$. 
Hence we obtain the formulas for $Q^+_{\vL}$ if we multiply $B_{\SC}$ on the both sides 
of (3.9.1) and (3.9.2) with respect to $P^+_{\vL}$.
Similarly, we obtain the formulas for $Q^-_{\vL}$ if we multiply ${}^tB_{\SC}$ 
on the both sides of (3.9.1) and (3.9.2) with respect to $P^-_{\vL}$. 
Thus we obtain the assertion for $Q^{\pm}_{\vL}$.  The theorem is proved.      
\end{proof}

\para{3.10.}
Since $\{ P^{\pm}_{\vL} \mid \vL \in Z_{n,D}\}$ and 
$\{ s_{\vL} \mid \vL \in Z_{n,D}\}$ give bases of $\Xi^{n,D}_{\BQ}(\Bt)$, 
there exist unique functions $K^{\pm}_{\vL,\vL'}(\Bt) \in \BQ(\Bt)$ satisfying the
formula;
\begin{equation*}
\tag{3.10.1}
s_{\vL}(x) = \sum_{\vL' \in Z_{n,D}}K^{\pm}_{\vL,\vL'}(\Bt)P^{\pm}_{\vL'}(x;\Bt).
\end{equation*} 
$K^{\pm}_{\vL,\vL'}(\Bt)$ is called the Kostka function associated to complex reflection 
groups, or associated to $r$-symbols, in short. 
It is known that the original Hall-Littlewood functions $P_{\la}$ associated to partitions 
$\la$ have similar properties as in Theorem 3.9.  
Hence by the uniqueness, our Kostka function coincides with Kostka polynomial in the 
case where $r = 1, e = 0, \Bs = \bold 0$. 
In the case where $r = 2$, we have $K^+_{\vL,\vL'}(\Bt) = K^-_{\vL,\vL'}(\Bt)$ 
since $q^+_{\vL} = q^-_{\vL}$ and so $P^{+}_{\vL} = P^-_{\vL}$. In this case, 
we denote $K^{\pm}_{\vL,\vL'}(\Bt)$ as $K_{\vL,\vL'}(\Bt)$. 
\par
The following properties are easily deduced from Theorem 3.9 (ii).
\par\medskip\noindent
(3.10.2) \ $K^{\pm}_{\vL,\vL}(\Bt) = 1$, $K^{\pm}_{\vL,\vL'}(\Bt) = 0$ for 
$\vL' \sim \vL,\vL' \ne \vL$.  Moreover, $K^{\pm}_{\vL,\vL'}(\Bt) = 0$ if $\vL' \gv \vL$.
\par\medskip
In the one-parameter case, we define a modified Kostka function 
$\wt K^{\pm}_{\vL,\vL'}(t) \in \BQ(t)$ by

\begin{equation*}
\tag{3.10.3}
\wt K^{\pm}_{\vL,\vL'}(t) = t^{a(\vL')}K^{\pm}_{\vL,\vL'}(t\iv).
\end{equation*}

\remarks{3.11.} \
(i)  Our definition of Hall-Littlewood functions and Kostka functions 
depend on the choice of a total order $\lve$ on $Z_{n,D}$. 
It is not known, in general, whether or not 
those functions actually depend on the total order.
Also our Kostka function $K^{\pm}_{\vL, \vL'}(\Bt)$ are rational functions on $\Bt$.
It is not known in general whether $K^{\pm}_{\vL,\vL'}(\Bt) \in \BZ[\Bt]$ or not.
\par
(ii) In [S4], the notion of limit symbols was introduced, which is defined 
by a special choice of $e$ and $\Bs$ so that each similarity class consists of
a single element. In that case (for the multi-parameter, with $D = \{ d^{\bullet}\}$), 
it is proved in [S6] that $P^{\pm}_{\vL}(x;\Bt) \in \Xi^{n,\Bd^{\bullet}}[\Bt]$
and that $K^{\pm}_{\vL, \vL'}(\Bt) \in \BZ[\Bt]$. Moreover, those functions do not
depend on the choice of the total order whenever it is compatible with 
the dominance order on $\SP_{n,r}$.  
\par
(iii) In the limit symbol case, it is proved in [S6] that 
our Kostka function $K^-_{\vL,\vL'}(\Bt)$ 
coincides with the function introduced by Finkelberg-Ionov [FI].  By using their 
results, it is shown that $K^-_{\vL, \vL'}(\Bt)$ has a geometric realization 
in terms of coherent sheaves, which implies that $K^-_{\vL, \vL'}(\Bt)$ is 
a polynomial in $t_1, \dots, t_r$ with non-negative coefficients. 
\par\bigskip
%%%%
%%%%
\section{Algorithm of computing Kostka functions}

\para{4.1.}
In [S2], it was shown that there exists an algorithm of 
computing Kostka functions associated to $Z_{n, D}$ with 
$D = \{ \Bd^{\bullet}\}$. In [S6], 
multi-parameter Kostka functions associated to limit symbols 
were introduced, but the algorithm of computing them was 
not discussed there. 
In this section, we show 
that a similar algorithm as in [S2] holds also for our Kostka functions 
(with multi-parameter) associated to $Z_{n, D}$ for any $D$.  
\par
Let $W_{n,r}$ be the complex reflection group $S_n \ltimes (\BZ/r\BZ)^n$. 
It is known that the irreducible characters of $W_{n,r}$ are parametrized 
by $\SP_{n,r}$. We denote by $\x^{\Bla}$ the irreducible character of 
$W_{n,r}$ corresponding to $\Bla \in \SP_{n,r}$. 
\para{4.2.}
We consider the matrix $X_{\pm}(\Bt) = M(p, P^{\pm})$, where 
$p = \{ p_{\vL} \}$ is the basis of $\Xi^{n,D}_{\BC}(\Bt)$ consisting of 
power sum symmetric functions, and $P^{\pm} = \{ P^{\pm}_{\vL} \}$ 
are bases of $\Xi^{n,D}_{\BQ}(\Bt)$ consisting of Hall-Littlewood functions. 
Thus $X_{\pm}(\Bt) = (X^{\pm}_{\vL,\vL'}(\Bt))$, 
where $X^{\pm}_{\vL,\vL'}(\Bt) \in \BC(\Bt)$ such that 
\begin{equation*}
\tag{4.2.1}
p_{\vL}(x) = \sum_{\vL' \in Z_{n,D}}X^{\pm}_{\vL,\vL'}(\Bt)P^{\pm}_{\vL'}(x;\Bt).
\end{equation*} 
In the classical case, Green polynomials $Q_{\la,\mu}(t)$ are defined by 
$Q_{\la,\mu}(t) = t^{n(\mu)}X_{\la, \mu}(t\iv)$ (see [Mac, III, 7]), 
where $X_{\la,\mu}(t)$ is defined similarly to (4.2.1) with respect to 
$p_{\la}(x)$ and $P_{\mu}(x;t)$ for $\la, \mu \in \SP_n$.  Thus we call 
$X^{\pm}_{\vL,\vL'}(\Bt)$ the generalized Green functions associated to 
$r$-symbols. 
\par
By Proposition 3.6 (iii), we have $P^{\pm}_{\vL'}(x;\bold 0) = s_{\vL'}(x)$. 
By the Frobenius formula (see [Mac]), we know, for $\Bla, \Bmu \in \SP_{n',r}$,  that 
\begin{equation*}
p_{\Bla} = \sum_{\Bmu \in \SP_{n',r}}\x^{\Bmu}(w_{\Bla})s_{\Bmu},  
\end{equation*}
where $w_{\Bla}$ is a representative of the conjugacy class in $W_{n',r}$ 
corresponding to $\Bla \in \SP_{n',r}$. 
(If $w \in W_{n',r}$ is conjugate to $w_{\Bla}$, we say that $w$ has type $\Bla$). 
Hence by the construction of $p_{\vL}$ and $s_{\vL'}$ in terms of the 
isomorphism $\vT_{\Bd}$, we have the following.
\par\medskip\noindent
(4.2.2) \ Assume that $\vL, \vL'  \in Z_{n,D}$. 
If $d(\vL) \ne d(\vL')$, then $X^{\pm}_{\vL,\vL'}(\bold 0) = 0$. 
If $d(\vL) = d(\vL') = \Bd$ and 
$\vL = \vL(\Bla), \vL' = \vL(\Bmu)$ with $\Bla, \Bmu \in \SP_{n',r}$ 
(under the isomorphism 
$\vT_{\Bd} : \wh\Xi^{n',\bold 0}_{\BQ}(\Bt) \simeq \Xi^{n,\Bd}_{\BQ}(\Bt)$),
then 
\begin{equation*}
X^{\pm}_{\vL,\vL'}(\bold 0) = \x^{\Bmu}(w_{\Bla}). 
\end{equation*} 
\par
In particular, $X_{\pm}(\bold 0)$ (after a suitable permutation of indices) 
turns out to be a diagonal block matrix with respect to the 
partition $Z_{n,D} = \coprod_{\Bd}Z_{n,\Bd}$, and the diagonal block corresponding 
to $Z_{n,\Bd}$ is exactly the character table of $W_{n',r}$. 
Since $X_{\pm}(\bold 0)$ does not depend on the signature, we denote it by 
$X(\bold 0)$. 
\par
On the other hand, by (3.7.4) and (3.8.1), we have
\begin{equation*}
\tag{4.2.3}
\sum_{\vL,\vL' \in Z_{n,D}}z_{\vL,\vL'}(\Bt) p_{\vL}(x)\ol p_{\vL}(y) = 
        \sum_{\vL, \vL' \in Z_{n,D}}b_{\vL',\vL}(\Bt)P^+_{\vL}(x;\Bt)P^-_{\vL'}(y;\Bt).
\end{equation*} 
Put $D(\Bt) = (b_{\vL',\vL}(\Bt))$.  Then $D(\Bt)$ is a diagonal block matrix 
with respect to the similarity classes in $Z_{n,D}$. 
Put $Z(\Bt) = (z_{\vL,\vL'}(\Bt))$. $Z(\Bt)$ is a non-singular matrix since 
$Z(\Bt_0)$ is non-singular by Remark 2.7.  It followed from the definition that
\par\medskip\noindent
(4.2.4) \  $Z(\Bt)$ is a block diagonal matrix with respect to 
the partition $Z_{n,D} = \coprod_{\Bd}Z_{n,\Bd}$.
\par\medskip  
By substituting the formula for $p_{\vL}(x)$ in (4.2.1) into 
the left hand side of (4.2.3), and by comparing the coefficients of 
$P^+_{\vL}(x;\Bt)P^-_{\vL'}(y;\Bt)$, we obtain the following equation of matrices.
\begin{equation*}
\tag{4.2.5}
{}^t X_+(\Bt)Z(\Bt) \ol X_-(\Bt) = D(\Bt),
\end{equation*}   
where $\ol X_-(\Bt)$ denotes the complex conjugate of the matrix $X_-(\Bt)$. 
(Note that $P^{\pm}_{\vL}(x;\Bt) \in \Xi^{n,D}_{\BQ}(\Bt)$.)
If we put $\vL(\Bt) = D(\Bt)\iv$, (4.2.5) can be rewritten as 
\begin{equation*}
\tag{4.2.6}
\ol X_-(\Bt)\vL(\Bt)\,{}^t\!X_{+}(\Bt) = Z(\Bt)\iv.
\end{equation*}
Let $K_{\pm}(\Bt) = M(s, P^{\pm})$ be the transition matrix between Schur functions and 
Hall-Littlewood functions, namely, 
\begin{equation*}
s_{\vL}(x) = \sum_{\vL' \in Z_{n,D}}K^{\pm}_{\vL,\vL'}(\Bt)P^{\pm}_{\vL'}(x;\Bt).
\end{equation*}
By Proposition 3.6, $K_{\pm}(\Bt)$ is a lower triangular block matrix, and 
all the diagonal blocks are identity matrices.  
Moreover, $K^{\pm}_{\vL,\vL'}(\Bt) \in \BQ(\Bt)$, and $K_{\pm}(\bold 0)$ is the identity 
matrix. 
\par
Since $M(s, P^{\pm}) = M(p,s)\iv M(p, P^{\pm})$, we have
$K_{\pm}(\Bt) = X(\bold 0)\iv X_{\pm}(\Bt)$. 
Substituting this formula into (4.2.6), we obtain 
\begin{equation*}
\tag{4.2.7}
K_-(\Bt)\vL(\Bt)\,{}^t\! K_+(\Bt) 
     = \ol X(\bold 0)\iv Z(\Bt)\iv \,{}^t\! X(\bold 0)\iv.
\end{equation*} 

Put $z_{\vL} = z_{\Bla}$ for $\vL = \vL(\Bla)$, and let $H$ be the diagonal matrix
whose $\vL\vL$-entry is $z_{\vL}\iv$. Note that $z_{\Bla} = |Z_{W'}(w_{\Bla})|$ 
for $W' = W_{n',r}$.  
Since $X(\bold 0)$ is the character table of $W_{n',r}$ on each $Z_{n,\Bd}$, 
by the orthogonality relations of irreducible characters of $W_{n',r}$, 
we have ${}^tX(\bold 0)H\ol X(\bold 0) = I$. Substituting this into (4.2.7), we have 
\begin{equation*}
\tag{4.2.8}
K_-(\Bt)\vL(\Bt)\,{}^t\! K_+(\Bt)  = {}^tX(\bold 0)HZ(\Bt)\iv H\ol X(\bold 0).
\end{equation*}  

\para{4.3.}
Put $Z(\Bt)\iv = (z'_{\vL,\vL'}(\Bt))$. Then $z'_{\vL,\vL'}(\Bt) = 0$ 
unless $d(\vL) = d(\vL')$. For $w,w' \in W_{n',r}$, we define 
$z'_{w,w'}(\Bt)$ as follows; assume that $w$ is of type $\Bla$ and 
$w'$ is of type $\Bmu$ for $\Bla,\Bmu \in \SP_{n',r}$, and  
then put $z'_{w,w'}(\Bt) = z'_{\vL,\vL'}(\Bt)$ for 
$\vL = \vL(\Bla), \vL' = \vL(\Bmu)$.
\par
Put $W' = W_{n',r}$.  
For $\Bla, \Bmu \in \SP_{n',r}$, we define $\w_{\Bla, \Bmu}(\Bt)$ by 
\begin{equation*}
\tag{4.3.1}
\w_{\Bla,\Bmu}(\Bt) = \frac{1}{|W '|^2}
                       \sum_{w, w' \in W '}\x^{\Bla}(w)\ol{\x^{\Bmu}(w')}
                         z'_{w,w'}(\Bt). 
\end{equation*} 
We define a matrix $\BOm = (\w_{\vL, \vL'}(\Bt))$ by 
\begin{equation*}
\tag{4.3.2}
\w_{\vL,\vL'}(\Bt) = \begin{cases}
                      \w_{\Bla,\Bmu}(\Bt) &\quad\text{ if } 
          \vL = \vL(\Bla), \vL' = \vL(\Bmu) \text{ with } \Bla,\Bmu \in \SP_{n',r}, \\
                      0  &\quad\text{ if } d(\vL) \ne d(\vL').
                     \end{cases}
\end{equation*}

Note that the matrix $Z(\Bt)\iv$ is computable by (2.6.3) and 
(2.9.1).  Hence the matrix $\BOm$ is computable if we know the 
character table of $W_{n',r}$ for various $n'$. 
\par
The following result gives an algorithm of computing $K_{\pm}(\Bt)$.
In the following, we consider the block matrices with respect to 
the similarity classes in $Z_{n,D}$. 

\begin{thm}   %%%%   Theorem 4.4
\begin{enumerate}
\item 
The matrix $K_{\pm}(\Bt)$ is determined as the unique solution of the 
following equation satisfying the properties {\rm (a)} and {\rm (b)}.
\begin{equation*}
\tag{4.4.1}
\BP_-\BLa\, {}^t\BP_+ = \BOm.
\end{equation*}
\begin{enumerate}
\item 
$\BP_{\pm}$ is a lower triangular block matrix, where the diagonal block is
the identity matrix. 
\item
$\BLa$ is a diagonal block matrix.
\end{enumerate}
Then $K_{\pm}(\Bt) = \BP_{\pm}$, and there exists an algorithm of 
computing $K_{\pm}(\Bt)$.
\item 
If $\vL, \vL' \in Z_{n,D}$ with $d(\vL) \ne d(\vL')$, 
then $(\BP_{\pm})_{\vL,\vL'} = 0$ and 
$\BLa_{\vL,\vL'} = 0$. 
\end{enumerate}
\end{thm}

\begin{proof}
We consider the matrix equation (4.2.8). 
It is clear that the right hand side of (4.2.8) 
coincides with the matrix $\BOm$ defined in 4.3. On the other hand, 
$\BP_{\pm} = K_{\pm}(\Bt), \vL(\Bt) = \BLa$ satisfies the
condition (a) and (b) in (i).   Thus $K_{\pm}(\Bt)$ and $\vL(\Bt)$ 
satisfies the equation (4.4.1). 
We now consider (4.4.1) as an equation of matrices, where $\BOm$ is known, and 
$\BP_{\pm}, \BLa$ are unknown matrices. By a standard argument from the linear 
algebra, we know that $\BP_{\pm}$ and $\BLa$ are the unique solution of 
(4.4.1) and that there exists an algorithm of computing $\BP_{\pm}$ and $\BLa$. 
Thus (i) holds. 
\par
Next we show (ii). 
The proof below is a natural generalization of the 
argument in the proof of [L4, Thm 24.4].  
We write $(\BP_{\pm})_{\vL,\vL'} = p^{\pm}_{\vL,\vL'}$ and 
$\BLa_{\vL,\vL'} = \la_{\vL, \vL'}$. 
Let $\BLa_{\SC}$ be the block corresponding to a similarity class $\SC$.
We consider the following statement.
\par\medskip\noindent
($A_{\SC}$) \ Any element $\la_{\vL,\vL'}$ in $\BLa_{\SC}$ satisfies (ii). 
\par\medskip\noindent
($B^{\pm}_{\SC}$) \ $p^{\pm}_{\vL, \vL'}$ satisfies (ii) for any $\vL' \in \SC$ 
and for any $\vL \in Z_{n,D}$.  
\par\medskip
We show, by induction, that ($A_{\SC}$) and ($B^{\pm}_{\SC}$) holds for any class $\SC$. 
\par
Take $\vL, \vL'\in Z_{n,D}$, such that $d(\vL) \ne d(\vL')$. 
Then (4.4.1) can be written as 
\begin{equation*}
\tag{4.4.2}
\sum_{\vL \gve \vL_1}\sum_{\vL' \gve \vL_2}
       p^-_{\vL,\vL_1}\la_{\vL_1,\vL_2}p^+_{\vL',\vL_2} = \w_{\vL,\vL'} = 0. 
\end{equation*}
First we show that
\par\medskip\noindent
(4.4.3) \ Assume that ($A_{\SC'}$) holds for $\SC' \lv \SC$, and 
($B^{\pm}_{\SC'}$) holds for $\SC' \lv \SC$. Then ($A_{\SC}$) holds. 
In particular, ($A_{\SC}$) holds if $\SC$ is the minimum class in $Z_{n,D}$.
\par\medskip 
Take $\vL, \vL' \in \SC$.
We note that 
\begin{equation*}
\tag{4.4.4}
\sum_{\substack{\vL \gve \vL_1 \\ \vL' \gve \vL_2}}
               \sum_{\substack{\vL \not\sim \vL_1 \\ \text{ or } \\
                      \vL' \not\sim \vL_2 }}
                 p^-_{\vL,\vL_1}\la_{\vL_1,\vL_2}p^+_{\vL',\vL_2} = 0
\end{equation*} 
under the assumption of (4.4.3).  In fact, 
suppose that the left hand side of (4.4.4) is not zero. 
Then there exists $\vL_1, \vL_2$
such that $p^-_{\vL,\vL_1} \ne 0, \la_{\vL_1, \vL_2} \ne 0, p^+_{\vL',\vL_2} \ne 0$.
If $\vL \not\sim \vL_1$, we have $d(\vL) = d(\vL_1)$ by (4.4.3).  
Since $\la_{\vL_1, \vL_2} \ne 0$, then $\vL_1 \sim \vL_2$.  Thus by (4.4.3), 
we have $d(\vL_1) = d(\vL_2)$ and $d(\vL_2) = d(\vL')$. 
This implies that $d(\vL) = d(\vL')$, which contradicts to our assumption. 
The case where $\vL' \not\sim \vL_2$ is proved similarly. 
Hence (4.4.4) holds.
\par
Substituting the equation (4.4.4) into (4.4.2), by noting that 
the diagonal block of $\BP_{\pm}$ corresponding to $\SC$ 
is the identity matrix, we have
\begin{equation*}
\tag{4.4.5}
\w_{\vL,\vL'} = \la_{\vL,\vL'} = 0.
\end{equation*}
Hence ($A_{\SC}$) holds. 
\par
Next we show that 
\par\medskip\noindent
(4.4.6) \ Assume that ($A_{\SC'}$) holds for any $\SC' \lve \SC$, and
($B^{\pm}_{\SC'}$) holds for any $\SC' \lv \SC$.  Then ($B^{\pm}_{\SC}$) holds.
\par\medskip
Take $\vL \in Z_{n,D}$ and $\vL' \in \SC$.  
Under the assumption of (4.4.6), the following equation can be proved 
in a similar way as in (4.4.4).
\begin{equation*}
\tag{4.4.7}
       \sum_{\vL \gve \vL_1}
         \sum_{\substack{\vL' \gve \vL_2 \\  \vL' \not\sim \vL_2 }}
                       p^-_{\vL,\vL_1}\la_{\vL_1, \vL_2}p^+_{\vL',\vL_2} = 0. 
\end{equation*}
By substituting (4.4.7) into (4.4.2), we have   
\begin{equation*}
\tag{4.4.8}
\sum_{\substack{\vL \gve \vL_1 }}
    \sum_{\substack{\vL' \gve \vL_2 \\ \vL' \sim \vL_2}}
        p^-_{\vL, \vL_1}\la_{\vL_1,\vL_2}p^+_{\vL',\vL_2} = 0. 
\end{equation*}
Since the diagonal block of $\BP_+$ corresponding to $\SC$ is the identity 
matrix, and since ($A_{\SC}$) holds, (4.4.8) can be written as
\begin{equation*}
\tag{4.4.9}
\sum_{\substack{\vL_1 \in \SC \\ d(\vL_1) = d(\vL')}}
                 p^-_{\vL, \vL_1}\la_{\vL_1, \vL'} = 0.
\end{equation*} 
This holds for any $\vL \in Z_{n,D}$ and $\vL' \in \SC$. 
For a fixed $\vL$, we consider (4.4.9) as a system of linear equations 
with unknown variables $p^-_{\vL,\vL_1}$.  Since $\BLa_{\SC}$ is a non-singular matrix, 
the submatrix $(\la_{\vL_1, \vL'})_{d(\vL_1) = d(\vL')}$ is also non-singular
(it gives a block decomposition of $\BLa_{\SC}$ under a suitable permutation of indices). 
It follows that $p^-_{\vL, \vL_1} = 0$ if $d(\vL) \ne d(\vL_1)$, and 
($B^-_{\SC}$) holds. 
If we take $\vL \in \SC$ and $\vL' \in Z_{n,D}$, by a similar argument as above, 
one can show that ($B^+_{\SC}$) holds. Hence (4.4.6) is proved. 
\par
Now ($A_{\SC}$) and ($B^{\pm}_{\SC}$) follows from (4.4.3) and (4.4.6). 
The assertion (ii) is proved. 
\end{proof}

\para{4.5.}
In the one-parameter case, we obtain a more precise result than 
Theorem 4.4. 
Hereafter, we consider the one-parameter case. 
\par
For $W = W_{n,r}$, 
we denote by $P_W(t)$ the Poincar\'e polynomial of $W$. 
Explicitly, it is given as
\begin{equation*}
\tag{4.5.1}
P_W(t) = \prod_{i = 1}^n\frac{t^{ir} - 1}{t-1}.
\end{equation*}
Let $\BV$ be the $n$-dimensional vector space over $\BC$ such that 
$\BV$ is the reflection representation of $W$. 
For any character $f$ of $W$, put 
\begin{equation*}
\tag{4.5.2}
R(f) = (t-1)^nP_W(t)|W |\iv\sum_{w \in W }
            \frac{\det_{\BV}(w)f(w)}{\det_{\BV}(t\cdot\id_{\BV} - w)}.
\end{equation*}
If we denote by $R_W = \bigoplus_iR_i$ the coinvariant algebra of $W$, 
we have $P_W(t) = \sum_{i \ge 0}(\dim R_i)t^i$, and 
$R(f) = \sum_i\lp R_i, f\rp t^i$ (the graded multiplicity of $f$ in $R_W$).   
Hence $R(f) \in \BZ[t]$. 
Let $N^*$ be the number of reflections in $W$. Then $N^*$ is the 
highest degree of $R_W$, and $R_{N^*} \simeq \ol{\det}_{\BV}$ as $W$-modules.  
Explicitly, we have
\begin{equation*}
\tag{4.5.3}
N^* = \sum_{i=1}^n(ri - 1) = rn(n+1)/2 - n.
\end{equation*}

For any $\Bla,\Bmu \in \SP_{n,r}$, we define $\wt\w'_{\Bla,\Bmu}(t) \in \BQ(t)$ by 
\begin{equation*}
\tag{4.5.4}
\wt\w'_{\Bla, \Bmu}(t) = \frac{t^n}{|W |}\sum_{w \in W }
       \frac{\x^{\Bla}(w)\ol\x^{\Bmu}(w)}{\det_{\BV}(t\cdot\id_{\BV} - w)}.
\end{equation*}
We define $\GG_W(t) \in \BZ[t]$ by $\GG_W(t) = (t-1)^nt^{N^*}P_W(t)$.
($\GG_W(t)$ is regarded as a formal analogue of the order of the 
finite classical group $G(\BF_q)$, 
where the role of the Weyl group is replaced by the complex reflection group 
$W_{n,r}$ (see [BMM]).)   
$\wt\w'_{\Bla,\Bmu}(t)$ can be written as 
\begin{equation*}
\wt\w'_{\Bla,\Bmu}(t) = \GG_W(t)\iv t^{N^* + n}
         R(x^{\Bla}\otimes \ol\x^{\Bmu}\otimes\ol{\det}_{\BV}).
\end{equation*} 
\par
For $\vL, \vL' \in Z_{n,D}$, we define $\wt\w_{\vL,\vL'}(t)$ as follows.
Put $W = W_{n,r}$. 
\begin{equation*}
\tag{4.5.5}
\wt\w_{\vL,\vL'}(t) = \begin{cases}
                   \GG_W(t)\wt\w'_{\Bla,\Bmu}(t) &\quad\text{ if }
                       \vL = \vL(\Bla), \vL' = \vL(\Bmu) \text{ with } 
                               \Bla, \Bmu \in \SP_{n',r}, \\
                   0     &\quad\text{ if } d(\vL) \ne d(\vL').
                   \end{cases}
\end{equation*}
Note that in the first case,  
$\w'_{\Bla,\Bmu}(t)$ is defined with respect to $W ' = W_{n',r}$. Since
$\GG_W(t)$ is divisible by $\GG_{W'}(t)$, we see that $\wt\w_{\vL,\vL'}(t) \in \BZ[t]$. 
\par
We define a matrix $\wt\BOm$ by $\wt\BOm = (\wt\w_{\vL,\vL'}(t))$. 
Note that $\wt\BOm$ can be computed explicitly once we know the character 
table of $W_{n',r}$ for various $n' \le n$. 
\par
Let $\wt K^{\pm}_{\vL,\vL'}(t)$ be the modified Kostka functions defined in 
(3.10.3).  Computing $K^{\pm}_{\vL,\vL'}(t)$ is equivalent to computing 
$\wt K^{\pm}_{\vL,\vL'}(t)$. 
We define a matrix $\wt K_{\pm}(t)$ by $\wt K_{\pm}(t) = (\wt K^{\pm}_{\vL,\vL'}(t))$. 
By a similar argument as in the proof of Theorem 4.4, we have the following result, 
which is a generalization of [S2, Thm. 5.4].  

%%%%%
%%%%%
\begin{thm}  %%%%  Theorem 4.6
\begin{enumerate}
\item 
The matrix $\wt K_{\pm}(t)$ is determined as the unique solution of the 
following equation satisfying the properties {\rm (a)} and {\rm (b)}.
\begin{equation*}
\tag{4.6.1}
\wt\BP_-\wt\BLa\, {}^t\wt\BP_+ = \wt\BOm.
\end{equation*}
\begin{enumerate}
\item 
$\wt\BP_{\pm}$ is a lower triangular block matrix, where the diagonal block is
the scalar matrix $t^{a(\vL)}I_{\SC}$ for the similarity class $\SC$ containing 
$\vL$.  
\item
$\wt\BLa$ is a diagonal block matrix.
\end{enumerate}
Then $\wt K_{\pm} = \wt\BP_{\pm}$.
\item 
If $\vL, \vL'  \in Z_{n,D}$ with $d(\vL) \ne d(\vL')$, 
then $(\wt\BP_{\pm})_{\vL,\vL'} = 0$ and 
$\wt\BLa_{\vL,\vL'} = 0$. 
\end{enumerate}
\end{thm}

\begin{proof}
We consider the formula (4.2.8), specialized to $\Bt = \Bt_0$.
We write $K_{\pm}(\Bt_0)$, $\vL(\Bt_0)$, $Z(\Bt_0)$ as
$K_{\pm}(t)$, $\vL(t)$, $Z(t)$, and write $X(\bold 0) = X(0)$.  
We will translate this formula to  the formula with respect to 
modified Kostka functions. 
It follows from the definition of $\wt K^{\pm}_{\vL,\vL'}(t)$, one can 
write as $\wt K_{\pm}(t) = K_{\pm}(t\iv)S$, where $S$ is the diagonal matrix 
whose $\vL\vL$-entry is $t^{a(\vL)}$. 
Then (4.2.8) is rewritten as 
\begin{equation*}
\tag{4.6.2}
\wt K_-(t)\vL'(t)\, {}^t\! \wt K_+(t) = \ol X(0)\iv HZ(t\iv)\iv H
                    \,{}^t\!X(0)\iv,
\end{equation*}
where $\vL'(t) = S\iv \vL(t\iv)S\iv$. 
Let $M = (M_{\vL,\vL'})$ be the right hand side of (4.6.2). 
As in the proof of Theorem 4.4, we see that $M_{\vL,\vL'} = 0$ if 
$d(\vL) \ne d(\vL')$. 
\par
Here we note the following relation. If $w_{\Bla}$ is an element in 
$W_{n',r}$ with type $\Bla \in \SP_{n',r}$, we have 
\begin{equation*}
\tag{4.6.3}
{\det}_{\BV}(t\cdot\id_{\BV} - w_{\Bla}) = 
      \prod_{k = 1}^r\prod_{j = 1}^{l(\la^{(k)})}(t^{\la^{(k)}_j} - \z^{k-1}). 
\end{equation*}
By comparing this with the definition of $z_{\Bla}(t)$ in 
(2.7.1) and (2.7.2), we have
\begin{equation*}
\tag{4.6.4}
z_{\Bla}(t\iv) = \frac{z_{\Bla}t^{n'}}{{\det}_{\BV}(t\cdot \id_{\BV} - w_{\Bla})}.
\end{equation*}
Assume that $\vL = \vL(\Bla), \vL' = \vL(\Bmu)$ 
with $\Bla,\Bmu \in \SP_{n',r}$. 
Note that $Z(t)$ is a diagonal matrix with $\vL\vL$-entry $z_{\vL}(t)\iv$.  
Substituting (4.6.4) into $Z(t\iv)\iv$, we can compute $M_{\vL, \vL'}$. 
By (4.6.2), we have
\begin{equation*}
M_{\vL,\vL'} = t^{n'}|W '|\iv \sum_{w \in W '}
       \frac{\x^{\Bla}(w)\ol\x^{\Bmu}(w)}{\det_{\BV}(t\cdot \id_{\BV} - w)} 
              = \wt\w'_{\Bla,\Bmu}(t).
\end{equation*}  
Hence if we put $\wt \vL(t) = \GG_W(t)\vL'(t)$, by (4.6.2), we obtain 
the required formula
\begin{equation*}
\wt K_-(t)\wt\vL(t)\,{}^t \!\wt K_+(t) = \wt\BOm. 
\end{equation*} 
Since $\wt K_{\pm}(t), \wt \vL(t)$ satisfy the condition (a), (b), 
the first assertion is proved. The assertion (ii) is 
proved in a similar way as in the proof of Theorem 4.4.
\end{proof}

\par\bigskip
%%%%
%%%%
\section{Generalized Green functions of Symplectic groups}

\para{5.1.}
Let $G$ be a connected reductive group over $\Bk$, where $\Bk$ is an algebraic closure 
of a finite field $\BF_q$ of $q$-elements, 
and $F : G \to G$ the Frobenius map corresponding to $\BF_q$, so that 
$G^F = G(\BF_q)$ is a finite reductive group. 
We denote by $G\uni$ the set of unipotent elements in $G$, and 
$G^F\uni$ the set of unipotent elements in $G^F$. 
Let $\SN_G$ be the set of pairs $(C, \SE)$, where $C$ is a unipotent class 
in $G$ and $\SE$ is a $G$-equivariant simple local system on $C$. 
We denote by $\SN_G^0$ the subset of $\SN_G$ consisting of $(C,\SE)$ such that 
$\SE$ is a cuspidal local system on $C$ ([L2]). 
Let $\ScS_G$ be the set of triples $(L, C_1, \SE_1)$, up to $G$-conjugate, 
where $L$ is a Levi subgroup of a parabolic subgroup $P$ of $G$, and 
$(C_1, \SE_1) \in \SN_L^0$. Take an $F$-stable $(L, C_1, \SE_1) \in \ScS_G$,  
namely, the triple such that $F(L) = L, F(C_1) = C_1$ and $F^*\SE_1 \simeq \SE_1$.  
We fix an isomorphism $\vf_1 : F^*\SE_1 \isom \SE_1$.  In [L3, 8.3], 
Lusztig 
constructed the semisimple perverse sheaf $K$ on $G$, associated to 
$(L,C_1, \SE_1)$, and the isomorphism $\vf : F^*K \isom K$ induced from $\vf_1$.   
He defined the generalized Green function $Q^G_{L,C_1,\SE_1, \vf_1}$ 
as the restriction to $G^F\uni$ of the characteristic function $\x_{K,\vf}$ on $G^F$,  
which is a $G^F$-invariant function on $G^F\uni$ with values in $\Ql$.
\par
In the case where $G = GL_n$, an $F$-stable maximal torus is written as 
$T_w$, twisted from the split maximal torus $T_0$ by $w \in S_n$.  
The $G^F$-conjugacy classes 
of $G^F\uni$ are parametrized by $\SP_n$, we denote by $u_{\nu} \in G^F\uni$ 
a representative of the class corresponding 
to $\nu \in \SP_n$.  
As the special case of $Q^G_{L,C_1,\SE_1, \vf_1}$, 
we have the Green function $Q^G_{T_w}$  on $G^F\uni$, which satisfies the relation
\begin{equation*}
\tag{5.1.1}
Q^G_{T_{w_{\mu}}} (u_{\nu}) =  Q_{\mu,\nu}(q) = 
\sum_{\la \in \SP_n}\x^{\la}(w_{\mu} )\wt K_{\la\nu}(q),
\end{equation*}
where $\x^{\la}$ is the irreducible character of $S_n$ corresponding to 
$\la \in \SP_n$, and $Q_{\mu,\nu}(t)$ is as in 4.2. 
\par
In view of the discussion in 4.2, (5.1.1) gives a combinatorial description of 
Green functions $Q^G_{T_w}$ in terms of Green polynomials $Q_{\mu,\nu}(t)$, 
or equivalently, of Kostka polynomials $K_{\la\nu}(t)$.  
In this section, we show, in the case of symplectic groups, that generalized 
Green functions can be described in a combinatorial way in terms of 
our generalized Green functions defined in 4.2, or equivalently, of Kostka functions 
$K_{\vL,\vL'}(t)$.

\para{5.2.}
From now on, assume that $G$ is the symplectic group $Sp_{2n}$
(with any characteristic). 
In this case, all the pairs $(C,\SE) \in \SN_G$ are $F$-stable.  
It is known that the set $\SN_G$ is described in terms of 2-symbols 
([L2] for $\ch \Bk \ne 2$, [LS] for $\ch\Bk = 2$).
We consider the 2-symbols $Z^{e,\Bs}_{n,D}$ as follows.
Put $e = 2, \Bs = (0,1)$ in the case where $\ch\Bk \ne 2$, 
and put $e = 4, \Bs = (0,2)$ in the case where $\ch\Bk = 2$.
In either case assume that $\a = 1$, namely, $m^{\bullet} = (m + 1, m)$, 
and let $D$ be the set of defects $\Bd = (d,0)$ or 
$(0,d)$ such that $d$ is odd.  
Then $Z^{e,\Bs}_{n,D}$ coincides with 
the set $X^{1,1}_n$ (resp. $X^{2,2}_n$) if $\ch\Bk \ne 2$ (resp. $\ch\Bk = 2$) 
in the notation in [LS]. 
Hereafter we write $Z^{e,\Bs}_{n,D}$ as $Z_{n,D}$. 
It is proved in [L2], [LS]  that 
$\SN_G$ is in bijection with $Z_{n,D}$ such that the unipotent classes 
in $G\uni$ correspond to the similarity classes in $Z_{n,D}$. 
Moreover, the set $D$ is in bijection with the set $\ScS_G$. 
For $(L, C_1, \SE_1) \in \ScS_G$ with $L \simeq (GL_1)^{n'} \times Sp_{2n - 2n'}$, 
$N_G(L)/L \simeq W_{n',2}$, there exists a natural bijection 
$Z_{n,\Bd} \isom W_{n',2}\wg$, which we denote by $\vL \mapsto \x^{\vL}$, 
such that the following generalized Springer correspondence holds 
([L2, Thm. 12.3], [LS, Thm. 2.4]).  
\begin{equation*}
\tag{5.2.1}
\coprod_{(L,C_1,\SE_1) \in \ScS_G}(N_G(L)/L)\wg \simeq 
       \coprod_{\Bd \in D}Z_{n, \Bd} = Z_{n,D}. 
\end{equation*}

\para{5.3.}
Let $(L, C_1, \SE_1) \in \ScS_G$ be corresponding to $\Bd \in D$. 
Put $\SW_L = N_G(L)/L$ and let $Z(L)$ be the center of $L$. 
Then $K$ can be decomposed into simple components
\begin{equation*}
\tag{5.3.1}
K \simeq \bigoplus_{\vL \in Z_{n,\Bd}}V_{\vL}\otimes A_{\vL},
\end{equation*}
where $V_{\vL}$ is an irreducible $\SW_L$-module with character $\x^{\vL}$, 
and $A_{\vL}$ is a simple perverse sheaf on $G$.
It is known that $A_{\vL}|_{G\uni} \simeq \IC((\ol C, \SE)[a_{\vL}]$, 
where $a_{\vL} = - \dim C - \dim Z(L)$, under the correspondence 
$\vL \lra (C,\SE)$.
By (5.3.1), $\vf : F^*K \isom K$ induces a canonical 
isomorphism $\vf_{\vL} : F^*A_{\vL} \isom A_{\vL}$ for $\vL \in Z_{n,\Bd}$
(which corresponds to $\f_{A_i}$ in [L4, 24.2]). 
We now define a function 
$X_{\vL}$ on $\ol C^F$ by 
\begin{equation*}
\tag{5.3.2}
X_{\vL}(g) = \sum_a(-1)^{a + a_{\vL}}\Tr(\vf_{\vL}, \SH^a_gA_{\vL})
      q^{-(a_{\vL} + r_{\vL})/2},  
\end{equation*}
for $g \in \ol C^F$, where $r_{\vL} = \dim \supp A_{\vL}$.
($X_{\vL}$ corresponds to $X_i$ in [L4, 24.2]). 
Note that $a_{\vL} + r_{\vL} = (\dim G - \dim C) - (\dim L - \dim C_1)$.  
\par
There exists an $F$-stable parabolic 
subgroup $P_0$ and an $F$-stable Levi subgroup $L_0$ of $P_0$, and 
$F$ acts trivially on $\SW_L = N_G(L_0)/L_0 \simeq W_{n',2}$.  Moreover any $F$-stable 
Levi subgroup $L$ can be obtained as $L_w$, by twisting $L_0$ by $w \in \SW_L$.  
For a pair $(C,\SE) \in \SN_G$ corresponding to $\vL \in Z_{n, \Bd}$, 
the  generalized Green function $Q^G_{L_w, C_1, \SE_1, \vf_1}$ 
can be written as   
\begin{equation*}
\tag{5.3.3}
Q^G_{L_w, C_1,\SE_1, \vf_1} = 
              \sum_{\vL \in Z_{n,\Bd}}\x^{\vL}(w)X_{\vL}q^{(a_{\vL} + r_{\vL})/2}.
\end{equation*} 
Thus the determination of $Q^G_{L,C_1,\SE_1,\vf_1}$ is reduced to 
the determination  
of various functions $X_{\vL}$ on $G^F\uni$ for $\vL \in Z_{n,\Bd}$. 

\para{5.4.}
Since $\SH^{a_{\vL}}A_{\vL}|_{G\uni} \simeq \SE$, 
$q^{-(a_{\vL} + r_{\vL})}\vf_{\vL}: F^*\SH^{a_{\vL}}A_{\vL} \to \SH^{a_{\vL}}A_{\vL}$ 
induces an isomorphism  $F^*\SE \isom \SE$, which we denote by $\psi_{\vL}$.  
It is known (under a suitable choice of $\vf_1$) 
that $\psi_{\vL}$ induces an automorphism of finite order $\SE_g \to \SE_g$ 
for any $g \in C^F$. We denote by $Y_{\vL}$ the characteristic function 
$\x_{\SE, \psi_{\vL}}$ on $C^F$, extended to the function on $G^F\uni$, namely 
\begin{equation*}
Y_{\vL}(g) = \begin{cases}
                \Tr(\psi_{\vL}, \SE_g) &\quad\text{ if } g \in C^F, \\
                0                  &\quad\text{ if } g \not\in C^F.
             \end{cases}
\end{equation*}
Then $\{ Y_{\vL} \mid \vL \in Z_{n,D}\}$ gives rise to a basis of the $\Ql$-space
of $G^F$-invariant functions of $G^F\uni$. 
We express $X_{\vL}$ as $X_{\vL} = \sum_{\vL' \in Z_{n,D}}p_{\vL, \vL'}Y_{\vL'}$
with $p_{\vL,\vL'} \in \Ql$, and consider the matrix 
$P = (p_{\vL,\vL'})$ indexed by $\vL, \vL' \in Z_{n,D}$.
We define a total order $\lve$ on $Z_{n,D}$ such that each similarity class 
gives an interval, and that $\vL' \lv \vL$ if $C' \subset \ol C$, where 
$C$ (resp. $C'$) corresponds to the similarity class containing $\vL$ (resp. $\vL'$).
(Note that in our case $a(\vL)$ coincides with $(\dim G\uni - \dim C)/2$ by 
[L5, 4.4],  
this agrees with the choice of $\lve$ in 3.5.)    
The following theorem was proved by Lusztig.  In the following, we consider 
the matrices with respect to the total order $\lve$, and regard them as block 
matrices according to the partition of $Z_{n,D}$ by similarity classes.  
%%%%
%%%%
\begin{thm}[{Lusztig [L4, Thm. 24.4]}]
The matrix $P$ is obtained as the unique solution of the 
equation 
\begin{equation*}
P\vL\,{}^t\!P = \Om.
\end{equation*}
Here $P$ is a lower triangular block matrix, where the diagonal block 
is the identity matrix, and $\vL$ is a diagonal block matrix. 
Moreover, $\Om = (\w_{\vL,\vL'})$ is defined as 
\begin{equation*}
\tag{5.5.1}
\w_{\vL,\vL'} = \frac{q^{-\dim G} q^{- (a_{\vL} + a_{\vL'})/2}|G^F|}
                 {|\SW_L|}\sum_{w \in \SW_L}
                   \frac{\x^{\vL}(w)\x^{\vL'}(w)}{|Z(L_w)^F|}
\end{equation*}
if $d(\vL) = d(\vL')$, and $\w_{\vL, \vL'} = 0$ if $d(\vL) \ne d(\vL')$. 
\end{thm}

\para{5.6.}
Theorem 5.5 shows that there is an algorithm of computing the matrix $P$. 
Thus the determination of $Q^G_{L,C_1, \SE_1, \vf_1}$ is reduced to the 
determination of the functions $Y_{\vL}$. Note that the definition of $Y_{\vL}$
depends on the choice of $\vf_1$. Even if $\vL \sim \vL'$ belong 
to the same unipotent class $C$, $Y_{\vL}$ and $Y_{\vL'}$ have no relations 
if $d(\vL) \ne d(\vL')$. However, we can show that there exists a canonical 
choice of $Y_{\vL}$. In order to describe $Y_{\vL}$, we need to choose 
a good representative for each class $C^F$. In [S1], it was shown, in the 
case where $\ch \Bk \ne 2$, that there exists good representatives in $C^F$, called 
split elements, which behave well with respect to the computation 
of Green functions, namely the case where $L$ is a maximal torus $T$, 
$C_1$ is the identity class $\{ 1\}$, and $\SE_1$ is the constant sheaf $\Ql$. 
By using split elements, $Y_{\vL}$ can be computed for $\vL \in Z_{n, \Bd^{\bullet}}$.
In [S5], it was shown that those split elements behave well also for any 
triple $(L, C_1, \SE_1)$. Also the case of $\ch\Bk = 2$ was discussed there. 
\par
In the case of symplectic groups, if we fix $u \in C^F$, and put 
$A(u) = Z_G(u)/Z_G^0(u)$, then $A(u)$ is an abelian group on which $F$ 
acts trivially.  The $G^F$-classes of $C^F$ are parametrized by $A(u)$.  
We denote by $C^F_a$ the $G^F$-class in $C^F$ corresponding to $a \in A(u)$. Also 
any $G$-equivariant simple local system $\SE$ on $C$ is parametrized by 
$A(u)\wg$. We write $\SE_{\r}$ the local system on $C$ corresponding to 
$\r \in A(u)\wg$. We have the following result.
%%%%
%%%%
\begin{thm}[{[S5, Cor. 4.4]}]  %%%%   Theorem 5.7 
Assume that $G = Sp_{2n}$.  There exists a good representative 
for $C^F$, called the split element, satisfying the following properties.
For any $(L, C_1, \SE_1) \in \ScS^F_G$, take the split element $u_1 \in C_1^F$, 
and define $\vf_1 : F^*\SE_1 \isom \SE_1$ so that the induced map 
$(\SE_1)_{u_1} \to (\SE_1)_{u_1}$ is the identity map. 
Then for any $(C,\SE) \lra \vL \in Z_{n, \Bd}$, $\psi_{\vL} : F^*\SE \isom \SE$
satisfies the property that the induced map $\SE_u \to \SE_u$ is the identity 
map for the split element $u \in C^F$.  In particular, if $\SE = \SE_{\r}$ 
for $\r \in A(u)\wg$, we have  
\begin{equation*}
Y_{\vL}(g) = \begin{cases}
                 \r(a) &\quad\text { if }  g \in C^F_a, \\
                   0   &\quad\text{ if } g \notin C^F. 
             \end{cases} 
\end{equation*}
\end{thm}   

\para{5.8.}
We consider the Kostka functions associated to the 2-symbols $Z_{n,D}$.
As in 3.10, we write them  as $K_{\vL, \vL'}(t), \wt K_{\vL, \vL'}(t)$, etc. 
 by ignoring the signature. 
The following result gives a combinatorial description of 
generalized Green functions.

%%%%
%%%%
\begin{thm}   %%%%  Theorem 5.9.
For any $\vL \in Z_{n,\Bd}$, the function $X_{\vL}$ can be written as 
\begin{equation*}
X_{\vL} = \sum_{\vL' \in Z_{n,D}}q^{-a(\vL)}\wt K_{\vL, \vL'}(q)Y_{\vL'}.
\end{equation*}
\end{thm}  

\begin{proof}
Here $\SW_L \simeq W_{n',2}$.  If $w \in W_{n',2}$ has type 
$(\nu, \s) \in \SP_{n',2}$ with $\nu : \nu_1 \ge \cdots \ge \nu_k > 0$, 
$\s : \s_1 \ge \cdots \ge \s_{l} > 0$,  then we have
\begin{equation*}
|Z(L_w)^F| = \prod_{1 \le i \le k}(q^{\nu_i} - 1)\prod_{1 \le j \le l}(q^{\s_j} + 1).
\end{equation*} 
Hence $|Z(L_w)^F| = \det_{\BV}(q\cdot \id_{\BV} - w)$ by (4.6.3). 
Assume that $d(\vL) = d(\vL')$. 
Since $|G^F| = \GG_W(q)$ for $W = W_{n,2}$, 
by comparing (5.5.1) with (4.5.4) for $W_{n',2}$, we have 
\begin{equation*}
\w_{\vL,\vL'} = q^{-\dim G - n'}q^{-(a_{\vL} + a_{\vL'})/2}\GG_W(q)\wt\w'_{\vL,\vL'}(q).
\end{equation*} 
Since $a(\vL) = (\dim G - n - \dim C)/2$ for any $\vL$ belonging to the class $C$, 
we have 
\begin{equation*}
a_{\vL}/2 = - (\dim C + n')/2 = a(\vL) - (\dim G - n + n')/2.   
\end{equation*}
It follows that
\begin{equation*}
\tag{5.9.1}
q^{a(\vL) + a(\vL')}\w_{\vL,\vL'} = q^{-n}\wt\w_{\vL,\vL'}(q).
\end{equation*}
On the other hand, we have 
$\w_{\vL,\vL'} = \wt\w_{\vL,\vL'}(q) = 0$ if $d(\vL) \ne d(\vL')$.  
\par
Let $S$ be the diagonal matrix whose $\vL\vL$-entry equals to $q^{a(\vL)}$.
It follows from (5.9.1) that we have an equation 
$(SP)(q^n \vL) \,{}^t(SP) = \wt\BOm(q)$, which satisfies the same property 
as in (4.6.1), specialized to $t = q$.  
Thus by the uniqueness of the solution of (4.6.1), we have $SP = \wt K(q)$, 
namely, $p_{\vL, \vL'} = q^{-a(\vL)}\wt K_{\vL, \vL'}(q)$ 
for any $\vL, \vL'$. The theorem is proved.   
\end{proof}

\remarks{5.10.}
(i) \ Theorem 5.9 implies that $\wt K_{\vL, \vL'}(t)$, and so $K_{\vL,\vL'}(t)$, 
do not depend on the choice of the total order $\lve$  in 3.5.
\par
(ii) \ Theorem 5.9 also implies that $\wt K_{\vL,\vL'}(t) \in \BZ[t]$.
In fact, $p_{\vL,\vL'} = q^{-a(\vL)}\wt K_{\vL, \vL'}(q)$ is regarded 
as a rational function on $q$, namely, there exists a rational function 
$p_{\vL,\vL'}(t) \in \BQ(t)$
such that $p_{\vL,\vL'}(q)$ coincides with $p_{\vL,\vL'}$.
It is known by [L4, (24.5.2)] that $p_{\vL,\vL'}(q^s)$ is an integer for 
any $s \in \BZ_{\ge 0}$, which implies that $p_{\vL, \vL'}(t)$ 
is a polynomial in $t$. 
Thus we have $\wt K_{\vL, \vL'}(t) \in \BZ[t]$.
Note that in the case where $p$ is good, namely if $p \ne 2$, 
it was known by [L4, Thm. 24.8]
that $p_{\vL, \vL'}(q)$ is a polynomial in $q$.   
In the case where $p = 2$, our result gives a weaker version of [L4, Thm. 24.8]
(here the purity of the Frobenius eigenvalues are not assumed). 
\par
(iii) \ 
If $\deg \wt K_{\vL,\vL'} \le a(\vL')$, we have 
$K_{\vL, \vL'}(t) \in \BZ[t]$.  
It is likely that $K_{\vL, \vL'}(t)$ is a polynomial 
with degree $a(\vL') - a(\vL)$.

\para{5.11.}
Under the setting in 5.1, we consider the Green function 
$Q^G_{T, \vf_1}$ on $G^F\uni$, which is a special case of the generalized Green 
functions $Q^G_{L, C_1, \SE_1, \vf_1}$, where $L = T$ is an 
$F$-stable maximal torus, $C_1 = \{ e \}$, $\SE_1$ is the constant
sheaf $\Ql$ on $C_1$, and $\vf_1 : F^*\Ql \isom \Ql$ is the 
canonical isomorphism. For any integer $r \ge 1$, we can also
consider $\vf_1^r: (F^r)^*\Ql \isom \Ql$, and the Green function 
$Q^G_{T, \vf_1^r}$ on $G^{F^r}$.  The following result was 
recently proved by M. Geck [G] in a full generality.

\begin{thm}[{[G]}]  %%%%   Theorem 5.12.
Let $G$ be a connected reductive group defined over $\BF_q$, and 
$u \in G^F$ a unipotent element.  Assume that $r$ is a prime
such that $r$ does not divide the order of $G^{F^r}$.  Then 
\begin{equation*}
Q^G_{T,\vf_1}(u) \equiv Q^G_{T, \vf_1^r}(u) \mod r.
\end{equation*}
\end{thm}

\para{5.13.}
We consider a similar problem for the generalized Green function
$Q^G_{L,C_1, \SE_1, \vf_1}$ on $G^F\uni$ for the case of symplectic 
groups. 
Fix an $F$-stable triple $(L, C_1, \SE_1) \in \ScS_G$ as in 5.1.
We choose a split element $u_1 \in C_1^F$ and an isomorphism 
$\vf_1 : F^*\SE_1 \isom \SE_1$ as in Theorem 5.7.  
Note that by [S5], if $u_1$ is split for $C_1^F$, then 
$u_1$ is also split for $C_1^{F^r}$ for any $r \ge 1$. 
We define $\vf_1^r: (F^r)^*\SE_1 \isom \SE_1$ by 
\begin{equation*}
\vf_1^r = \vf\circ F^*(\vf_1) \circ \cdots \circ (F^*)^{r-1}(\vf_1)
    : (F^*)^r\SE_1 \isom \cdots \isom (F^*)^2\SE_1 \isom F^*\SE_1 \isom \SE_1.
\end{equation*}
Then $(L, C_1, \SE_1)$ is also $F^r$-stable, and one can consider 
the generalized Green function $Q^G_{L, C_1, \SE_1, \vf_1^r}$ on $G^{F^r}\uni$. 
Here note that for any $u \in G^F\uni$, $Q^G_{L,C_1, \SE_1, \vf_1}(u) \in \BZ$. 
This follows from the fact that $p_{\vL,\vL'}(q) \in \BZ[q]$, and the description 
of $Y_{\vL}$ in Theorem 5.7. We have the following result.

\begin{thm} %%%%  Theorem 5.14.
Assume that $G$ is a symplectic group. 
Let $r$ be a prime such that $r$ does not divide the order of $G^{F^r}$. 
For any $u \in G^F\uni$, we have 
\begin{equation*}
Q^G_{L,C_1, \SE_1, \vf_1}(u) \equiv Q^G_{L, C_1, \SE_1, \vf_1^r}(u)  \mod r.
\end{equation*}
\end{thm} 

\begin{proof}
Take $(C,\SE)$ and let $\vL$ be the corresponding symbol. 
Here $(C,\SE)$ is $F$-stable. 
We choose a split element $u\in C^F$, and define the function  
$Y_{\vL}$ on $G^F\uni$ given as in Theorem 5.7. We can also 
define the corresponding function on $G^{F^r}\uni$ by using the same 
$u$, which we denote by $Y^{(r)}_{\vL}$.  We note that
\par\medskip\noindent
(5.14.1) \  Assume that $r$ is odd.  Then for any $u \in C^F$, we have
$Y_{\vL}(u) = Y^{(r)}_{\vL}(u)$. 
\par\medskip
Here  $F$ acts trivially on $A_G(u)$, and 
for each $a \in A_G(u)$,  let $u_a$ be a representative of the $G^F$-class in 
$C^F$ corresponding to $a \in A_G(u)$.  
Then $u_a$ is obtained as $u_a = gug\iv$ for some $g \in G$ such that  
$\a = g\iv F(g) \in Z_G(u)$ is a lift of $a \in A_G(u) = Z_G(u)/Z_G^0(u)$ on $Z_G(u)$.   
Here we have
\begin{equation*}
g\iv F^r(g) = \a F(\a) \cdots F^{r-1}(\a) \in Z_G(u).
\end{equation*}
Since $F$ acts trivially on $A_G(u)$, the image of $g\iv F^r(g)$ 
coincides with $a^r \in A_G(u)$ under the map 
$Z_G(u) \to A_G(u)$. 
As $A_G(u) \simeq (\BZ/2\BZ)^c$ for some $c$, 
we see that $a^r = a$ in $A_G(u)$.  (5.14.1) follows from this. 
\par
Let $X_{\vL}$ be the function on $G^F\uni$ as in 5.3, and 
we denote by $X_{\vL}^{(r)}$ the corresponding function on $G^{F^r}\uni$.
By Theorem 5.9, we can write them as 
\begin{align*}
\tag{5.14.2}
X_{\vL} &= \sum_{\vL'}p_{\vL,\vL'}(q)Y_{\vL'}, \\
X^{(r)}_{\vL} &=  \sum_{\vL'}p_{\vL, \vL'}(q^r)Y^{(r)}_{\vL'}.  
\end{align*} 
Since $p_{\vL,\vL'}(t) \in \BZ[t]$, we have 
$p_{\vL,\vL'}(q) \equiv p_{\vL,\vL'}(q^r) \mod r$ by Fermat's little theorem. 
If $r$ satisfies the condition in the theorem, then $r$ is odd, and so 
(5.14.1) holds.  Thus by (5.14.2), we see that 
$X_{\vL}(u) \equiv X^{(r)}_{\vL}(u) \mod r$.   
The theorem now follows from (5.3.3).  
\end{proof}

\remark{5.15.}
Let $G$ be a connected reductive group, and 
$Q^G_{L,C_1, \SE_1, \vf_1}$ be the generalized Green function given in 5.1.
For each $F$-stable pair $\vL = (C,\SE) \in \SN_G$, the function $X_{\vL}$ can 
be defined by a similar formula as in (5.3.2), and (5.3.3) holds under 
a suitable modification (here and below we use symbols $\vL, \vL'$, etc. just to denote  
the elements in $\SN_G$). Lusztig's theorem (Theorem 5.5) gives, under this general
situation, an algorithm of computing $P = (p_{\vL,\vL'})$.  We note that
\par\medskip\noindent
(5.15.1) \ There exists a polynomial $p_{\vL, \vL'}(t) \in \BZ[t]$
such that $p_{\vL,\vL'}(q^r)$ coincides with $p_{\vL,\vL'}$ for $G^{F^r}$,
under the choice of $r$ satisfying a certain congruence condition. 
\par\medskip
In fact, $\Om$ can be defined in general by a similar formula as in (5.5.1), 
and it is regarded as a matrix of rational functions on $t$ evaluated at $q^r$ for 
$r$ satisfying the congruence condition. Then the equation $P\vL\,{}^t\!P = \Om$
determines $P$ uniquely in the level of rational functions, as far as this 
algorithm works well, namely if the division by 0 does not occur in each step.
But this is guaranteed by Theorem 5.5 which asserts that the evaluation at $t = q^r$ 
gives an algorithm  of computing $P$.         
Thus $p_{\vL,\vL'}(t) \in \BQ(t)$.  But we know by [L4, (24.5.2)] 
that $p_{\vL,\vL'} \in \BZ$. Hence (5.15.1) holds. 
\par
By using this result instead of Theorem 5.9, 
we can extend Theorem 5.14 to the case of special orthogonal groups as follows.
Let $G$ be a special orthogonal group. 
Then thanks to [S5, Cor. 4.4], a similar formula as in (5.14.1) holds for $G$ 
if $r$ is prime to the order of $G^{F^r}$.  Thus combined with (5.15.1), we have
\par\medskip\noindent
(5.15.2) \  Let $G$ be a special orthogonal group. 
Then the generalized Green function $Q^G_{L,C_1, \SE_1, \vf_1}$ 
satisfies a similar formula as in Theorem 5.14.

\par

\par\vspace{1cm}
\noindent
T. Shoji \\
School of Mathematical Sciences, Tongji University \\ 
1239 Siping Road, Shanghai 200092, P. R. China  \\
E-mail: \verb|shoji@tongji.edu.cn|

\end{document}